\documentclass[12pt,reqno]{article}

\usepackage{xcolor}
\usepackage{amssymb}
\usepackage{amsmath}
\usepackage{amsthm}
\usepackage{amsfonts} 
\usepackage{amscd}
\usepackage{graphicx}

\usepackage[colorlinks=true,
linkcolor=webgreen,
filecolor=webbrown,
citecolor=webgreen]{hyperref}

\definecolor{webgreen}{rgb}{0,.5,0}
\definecolor{webbrown}{rgb}{.6,0,0}

\usepackage{color}
\usepackage{fullpage}
\usepackage{float}

\usepackage{latexsym}
\usepackage{accents}

\DeclareMathOperator{\Li}{Li}
\DeclareMathOperator{\sgn}{sgn}

\begin{document}

\theoremstyle{plain}
\newtheorem{theorem}{Theorem}
\newtheorem{corollary}[theorem]{Corollary}
\newtheorem{lemma}{Lemma}
\newtheorem{example}{Example}
\newtheorem{remark}{Remark}

\begin{center}
\vskip 1cm{\LARGE\bf 
Series and sums involving the floor function \\
}
\vskip 1.cm
{\large
Kunle Adegoke \\
Department of Physics and Engineering Physics \\ Obafemi Awolowo University, Ile-Ife\\Nigeria \\
\href{mailto:adegoke00@gmail.com}{\tt adegoke00@gmail.com}

\vskip 0.2 in

Robert Frontczak\footnote{Statements and conclusions made in this article by Robert Frontczak are entirely those of the author. They do not necessarily reflect the views of LBBW.} \\
Landesbank Baden-W\"urttemberg, Stuttgart \\ Germany \\
\href{mailto:robert.frontczak@lbbw.de}{\tt robert.frontczak@lbbw.de}

\vskip 0.2 in

Taras Goy  \\
Faculty of Mathematics and Computer Science\\
Vasyl Stefanyk Precarpathian National University, Ivano-Frankivsk \\Ukraine\\
\href{mailto:taras.goy@pnu.edu.ua}{\tt taras.goy@pnu.edu.ua}}
\end{center}

\vskip .2 in

\begin{abstract}
Let $(a_n)_{n\geq 0}$ be an arbitrary sequence and $(a_{\lfloor n/k \rfloor})_{n\geq 0}$ 
its dual floor sequence. We study infinite series and finite generalized binomial sums involving $(a_{\lfloor n/k \rfloor})_{n\geq 0}$. 
As applications we prove a range of new closed form expressions for Fibonacci (Lucas) series and binomial sum identities as particular cases.
\vskip .05 in

\end{abstract}

\section{Motivation}

As usual, the Fibonacci numbers $F_n$ and the Lucas numbers $L_n$ are defined, for \text{$n\in\mathbb Z$}, through the recurrence relations $F_n = F_{n-1}+F_{n-2}$, $n\ge 2$, with initial values $F_0 = 0$, $F_1 = 1$ and $L_n = L_{n-1}+L_{n-2}$ with $L_0 = 2$, $L_1 = 1$. 

For negative subscripts we have $F_{-n} = (-1)^{n-1}F_n$ and $L_{-n}=(-1)^n L_n$. They possess the explicit formulas (Binet forms)
\begin{equation}\label{binet}
F_n = \frac{\alpha^n - \beta^n }{\alpha - \beta },\quad L_n = \alpha^n + \beta^n,\quad n\in\mathbb Z.
\end{equation}
For more information we refer to Koshy \cite{Koshy} and Vajda \cite{Vajda} who have written high quality books dealing with Fibonacci and Lucas numbers. 

The paper is motivated by the following series which appeared recently as a problem proposal \cite{Frontczak}:
\begin{equation*}
\sum_{n=0}^\infty \Big (\Big\lfloor \frac{n}{2} \Big\rfloor + 1 \Big ) \frac{F_n}{2^n} = \frac{32}{5} \qquad \mbox{and} \qquad 
\sum_{n=0}^\infty \Big (\Big\lfloor \frac{n}{2} \Big\rfloor + 1 \Big ) \frac{L_n}{2^n} = 16,
\end{equation*}
where $\lfloor x \rfloor$ denotes the floor function. 

Such series allow generalizations in various directions. A gibonacci version of the above series formulas is given by
\begin{equation*}
\sum_{n=0}^\infty \Big (\Big\lfloor \frac{n}{2} \Big\rfloor + 1 \Big ) \frac{G_n}{2^n} = \frac{8(3a+4b)}{5},
\end{equation*}
where $G_n$ denotes the gibonacci number defined by the recurrence $G_n = G_{n-1} + G_{n-2}$, $n \geq 2$, 
with $G_0 = a$ and $G_1 = b$ ($a$ and $b$ are arbitrary). Note that $F_n$ corresponds to the case of $G_n$ when $a = 0$ and $b = 1$, while $L_n$ is obtained when $a = 2$ and $b = 1$.

Many properties and applications of the floor and ceiling functions have been previously investigated. 
Here we refer to the contributions \cite{Furdui,Graham,Nyblom,Palka,Podrug,Somu}. Additionally, the analysis of series (both finite and infinite) with these functions has attracted recent attention \cite{Shah-1,Shah-2}.

In this paper, we consider the following general setting. Let $(a_n)_{n\geq 0}$ be an arbitrary sequence and $(a_{\lfloor n/k \rfloor})_{n\geq 0}$ its dual floor sequence. We study infinite series and finite generalized binomial sums involving $(a_{\lfloor n/k \rfloor})_{n\geq 0}$. As applications we prove a range of new closed form expressions for Fibonacci (Lucas) series and binomial sum identities as particular cases.

\section{Some infinite series}

The next fundamental lemma will be crucial in the first part of the study. 
It will enable us to establish many results for infinite series involving the floor function.
\begin{lemma}\label{fund_lem}
Let $k\geq 1$ be an integer and $z\in\mathbb{C}$ with $|z|<1$. Let further $(a_n)_{n\geq 0}$ be an arbitrary sequence with ordinary generating function $F(z)$. Then the ordinary generating function of the sequence 
$(a_{\lfloor n/k \rfloor})_{n\geq 0}$ is 
\begin{equation}\label{gf_floor_id1}
F^+(z,k) = \sum_{n=0}^\infty a_{\lfloor n/k \rfloor} z^n = \frac{1-z^k}{1-z} F(z^k).
\end{equation}
Also, the ordinary generating function of the sequence $((-1)^n a_{\lfloor n/k \rfloor})_{n\geq 0}$ is 
\begin{equation*}
F^{-}(z,k) = \sum_{n=0}^\infty a_{\lfloor n/k \rfloor} (-1)^n z^n = \frac{1+(-1)^{k+1}z^k}{1+z}F\big((-1)^k z^k\big).
\end{equation*}
\end{lemma}
\begin{proof} We have	
\begin{align*}
F^+(z;k) & =  \sum_{n=0}^\infty a_{\lfloor n/k \rfloor} z^n =  \sum_{n=0}^\infty a_n \sum_{j=0}^{k-1} z^{kn+j} \\
& =  \sum_{j=0}^{k-1} z^j \sum_{n=0}^\infty a_n z^{kn}  =  \frac{1-z^k}{1-z} F(z^k).
\end{align*}
The second identity follows from the first by replacing $z$ by $-z$.
\end{proof}
\begin{lemma}\label{example1}
For integer $k\geq 1$ and $|z|<1$ the following expressions are valid:
\begin{equation}\label{ex1_floor_id1}
\sum_{n=0}^\infty \Big\lfloor \frac{n}{k} \Big\rfloor z^n = \frac{z^k}{(1-z)(1-z^k)}
\end{equation}
and
\begin{equation}\label{ex1_floor_id2}
\sum_{n=0}^\infty (-1)^n \Big\lfloor \frac{n}{k} \Big\rfloor  z^n = \frac{(-1)^k  z^k}{(1+z)\big(1+(-1)^{k+1}z^k\big)}.
\end{equation}
\end{lemma}
\begin{proof}
Use Lemma \ref{fund_lem} with $a_n=n$ in conjunction with generating function
\begin{equation*}
F(z) = \sum_{n=0}^\infty n z^n = \frac{z}{(1-z)^2}.
\end{equation*}
The second identity follows from the first by replacing $z$ with $-z$.
\end{proof}
\begin{remark}
Since for $k<0$, $\left\lfloor\frac{n}{k}\right\rfloor = - 1 - \left\lfloor\frac{n-1}{-k}\right\rfloor$,
we can extend Formulas \eqref{ex1_floor_id1} and \eqref{ex1_floor_id2} to negative integer $k$ as follows:
\begin{gather*}
\sum_{n=0}^\infty \Big\lfloor \frac{n+1}{k} \Big\rfloor z^n = \frac{z^k}{(1-z)(1-z^k)},\\
\sum_{n=0}^\infty (-1)^n\Big\lfloor \frac{n+1}{k}  \Big\rfloor  z^n = \frac{(-1)^k  z^k}{(1+z)\big(1+(-1)^{k+1}z^k\big)}.
\end{gather*}
\end{remark}
\begin{remark} Since  
$\left\lfloor\frac{n}{k}\right\rfloor = \left\lceil\frac{n+1}{k}\right\rceil-1$ for $k\geq1$ and 
$\left\lfloor\frac{n+1}{k}\right\rfloor = \left\lceil\frac{n}{k}\right\rceil-1$ for $k\leq-1$, 
from \eqref{ex1_floor_id1} and \eqref{ex1_floor_id2} we have infinite series involving the ceiling function:
\begin{gather*}
\sum_{n=0}^\infty \Big\lceil \frac{n+1}{k} \Big\rceil z^n = \frac{1}{(1-z)(1-z^k)},\qquad k\geq 1,\\
\sum_{n=0}^\infty \Big\lceil \frac{n}{k} \Big\rceil  z^n = \frac{1}{(1-z)(1-z^k)}, \qquad k\leq -1, 
\end{gather*}
and
\begin{equation*}\label{ex1_ceil_id2_1}
\sum_{n=0}^\infty (-1)^n \Big\lceil \frac{n+1}{k} \Big\rceil z^n = \frac{1}{(1+z)\big(1-(-1)^{k}z^k\big)}, \qquad k\geq 1,
\end{equation*}
\begin{equation*}\label{ex1_ceil_id2_2}
\sum_{n=0}^\infty (-1)^n \Big\lceil \frac{n}{k} \Big\rceil z^n = \frac{1}{(1+z)\big(1-(-1)^{k}z^k\big)}, \qquad k\leq -1.
\end{equation*}
\end{remark}

Lemma \ref{example1} leads immediately to a range of new Fibonacci (Lucas) series evaluations. 
\begin{theorem}\label{thm_1}
For any integers $m$ and $s$, an integer $k\geq 1$ and any  $p>\alpha^s$, we have
\begin{equation*}
\sum_{n=0}^\infty \Big\lfloor \frac{n}{k} \Big\rfloor \frac{F_{sn+m}}{p^{n+1}} = 
\frac{p^k \big(p F_{sk+m}-(-1)^s F_{s(k-1)+m}\big) - (-1)^{sk} \big(pF_{m} -(-1)^s F_{m-s}\big)}
{\big(p^2-pL_s+(-1)^s\big)\big(p^{2k} - p^kL_{sk} + (-1)^{sk}\big)}
\end{equation*}
and
\begin{equation*}
\sum_{n=0}^\infty \Big\lfloor \frac{n}{k} \Big\rfloor \frac{L_{sn+m}}{p^{n+1}} = 
\frac{p^k \big(p L_{sk+m}-(-1)^s L_{s(k-1)+m}\big) - (-1)^{sk} (pL_{m} -(-1)^sL_{m-s})}
{\big(p^2-pL_s+(-1)^s\big)\big(p^{2k} - p^kL_{sk} + (-1)^{sk}\big)}.
\end{equation*}
\end{theorem}
\begin{proof} 
Inserting $z=\alpha^s/p$ in \eqref{ex1_floor_id1} gives
\begin{equation*}
\sum_{n=0}^\infty \Big\lfloor \frac{n}{k} \Big\rfloor \frac{\alpha^{sn}}{p^{n}} = \frac{p \alpha^{sk}}{(p-\alpha^s)(p^{k}-\alpha^{sk})}\quad {\text{or}}\quad \sum_{n=0}^\infty \Big\lfloor \frac{n}{k} \Big\rfloor \frac{\alpha^{sn+m}}{p^{n}} = \frac{p \alpha^{sk+m}}{(p-\alpha^s)(p^{k}-\alpha^{sk})}.
\end{equation*}
Similarly,
\begin{equation*}
\sum_{n=0}^\infty \Big\lfloor \frac{n}{k} \Big\rfloor \frac{\beta^{sn+m}}{p^{n}} = \frac{p \beta^{sk+m}}{(p-\beta^s)(p^k-\beta^{sk})}.
\end{equation*}
Now, we combine the results according to the Binet forms \eqref{binet} making use of $\alpha\beta=-1$ and $\alpha+\beta=1$.
\end{proof}
\begin{example} 
Taking particular values for the parameters $s$ and $p$ in Theorem \ref{thm_1} leads to following series valid for $k\geq 1$:
\begin{equation*} 
\sum_{n=0}^\infty \Big\lfloor \frac{n}{k} \Big\rfloor \frac{F_{n+m}}{2^{n+1}} = \frac{2^k F_{k+m+2} - (-1)^k F_{m+2}}{4^k - 2^k L_k + (-1)^k},\qquad 
\sum_{n=0}^\infty \Big\lfloor \frac{n}{k} \Big\rfloor \frac{L_{n+m}}{2^{n+1}} = \frac{2^k L_{k+m+2} - (-1)^k L_{m+2}}{4^k - 2^k L_k + (-1)^k},
\end{equation*}
\begin{equation*}
\sum_{n=0}^\infty \Big\lfloor \frac{n}{k} \Big\rfloor \frac{F_{n+m}}{3^{n+1}} = \frac{3^k L_{k+m+1} - (-1)^k L_{m+1}}{5(9^k - 3^kL_k + (-1)^k)},\qquad 
\sum_{n=0}^\infty \Big\lfloor \frac{n}{k} \Big\rfloor \frac{L_{n+m}}{3^{n+1}} = \frac{3^k F_{k+m+1} - (-1)^k F_{m+1}}{9^k - 3^kL_k  + (-1)^k},
\end{equation*}
\begin{equation*}
\sum_{n=0}^\infty \Big\lfloor \frac{n}{k} \Big\rfloor \frac{F_{2n+m}}{4^{n+1}} = \frac{4^k L_{2k+m+1} - L_{m+1}}{5(16^{k} - 4^k L_{2k} + 1)},\qquad
\sum_{n=0}^\infty \Big\lfloor \frac{n}{k} \Big\rfloor \frac{L_{2n+m}}{4^{n+1}} = \frac{4^k F_{2k+m+1} - F_{m+1}}{16^{k} - 4^k L_{2k} + 1},
\end{equation*}
\begin{equation*}
\sum_{n=0}^\infty \Big\lfloor \frac{n}{k} \Big\rfloor \frac{F_{2sn+m}}{L_{2s}^{n+1}} = \frac{L_{2s}^k F_{2s(k+1)+m} - F_{m+2s}}{L_{2s}^{2k} 
- L_{2s}^k L_{2sk} + 1}, \qquad
\sum_{n=0}^\infty \Big\lfloor \frac{n}{k} \Big\rfloor \frac{L_{2sn+m}}{L_{2s}^{n+1}} = \frac{L_{2s}^k L_{2s(k+1)+m} - L_{m+2s}}{L_{2s}^{2k} 
- L_{2s}^k L_{2sk} + 1},
\end{equation*}
\begin{equation*}
\sum_{n=0}^\infty \Big\lfloor \frac{n}{k} \Big\rfloor \frac{F_{2sn+m}}{F^{n+1}_p} 
= \frac{F^k_p\big(F_{p}F_{2sk+m}-F_{2s(k-1)+m}\big) - F_pF_m + F_{m-2s}}{\big(F_p^2 - F_p L_{2s}+1\big)\big(F_p^{2k} - F_p^k L_{2sk} + 1 \big)}, \quad F_p>\alpha^s. 
\end{equation*}
\end{example}

The alternating version of Theorem \ref{thm_1} is proved in exactly the same manner, the proofs are therefore omitted. 
\begin{theorem}\label{thm_2}
For any integers $m$ and $s$, integer $k\geq 1$ and any $p>\alpha^s$, we have
\begin{equation*}\label{main_Fib2}
\sum_{n=0}^\infty (-1)^{n-k} \Big\lfloor \frac{n}{k} \Big\rfloor \frac{F_{sn+m}}{p^{n+1}} 
= \frac{p^k\big(p F_{sk+m}+(-1)^sF_{s(k-1)+m}\big)-(-1)^{k(s-1)}\big(pF_{m}+(-1)^sF_{m-s}\big)}{\big(p^2+pL_s+(-1)^s\big)\big(p^{2k} -(-1)^k p^k L_{sk} + (-1)^{sk}\big)}
\end{equation*}
and
\begin{equation*}\label{main_Luc2}
\sum_{n=0}^\infty (-1)^{n-k} \Big\lfloor \frac{n}{k} \Big\rfloor \frac{L_{sn+m}}{p^{n+1}} 
= \frac{p^k\big( p L_{sk+m}+(-1)^sL_{s(k-1)+m}\big)-(-1)^{k(s-1)}\big(pL_{m}+(-1)^sL_{m-s}\big)}{\big(p^2+pL_s+(-1)^s\big)\big(p^{2k} -(-1)^k p^kL_{sk} + (-1)^{sk}\big)}.
\end{equation*}
\end{theorem}
\begin{example} 
Taking particular values for $s$ and $p$ in Theorem \ref{thm_2} leads to the following special series valid for $k\geq 1$:
\begin{gather*}
\sum_{n=0}^\infty (-1)^{n-k}\Big\lfloor \frac{n}{k} \Big\rfloor  \frac{F_{n+m}}{2^{n+1}} 
= \frac{2^k L_{k+m-1} - L_{m-1}}{5(4^k - (-2)^{k} L_k + (-1)^k)},\\
\sum_{n=0}^\infty(-1)^{n-k} \Big\lfloor \frac{n}{k} \Big\rfloor \frac{L_{n+m}}{2^{n+1}} 
= \frac{2^k F_{k+m-1} - F_{m-1}}{4^k - (-2)^{k+1} L_k + (-1)^k},\\
\sum_{n=0}^\infty(-1)^{n} \Big\lfloor \frac{n}{k} \Big\rfloor  \frac{F_{2n+m}}{3^{n+1}} 
= \frac{(-3)^k (3F_{2k+m} + F_{2(k-1)+m})-3F_m-F_{m-2}}{19 (9^k-(-3)^kL_{2k}+1)},\\
\sum_{n=0}^\infty(-1)^{n} \Big\lfloor \frac{n}{k} \Big\rfloor  \frac{L_{2n+s}}{3^{n+1}} 
= \frac{(-3)^k(3L_{2k+s} + L_{2(k-1)+s})-3L_s-L_{s-2}}{19(9^k-(-3)^kL_{2k}+1)},\\
\sum_{n=0}^\infty (-1)^{n-k} \Big\lfloor \frac{n}{k} \Big\rfloor \frac{F_{2sn+m}}{L_{2s}^{n+1}} 
= \frac{L_{2s}^{k}( F_{2sk+m}L_{2s} + F_{2s(k-1)+m}) - (-1)^k(L_{2s}F_{m}+F_{m-2s})}
{(2L_{2s}^2 + 1)\big(L_{2s}^{2k} - (-1)^{k} L_{2s}^k L_{2sk} + 1\big)},\\
\sum_{n=0}^\infty (-1)^{n-k} \Big\lfloor \frac{n}{k} \Big\rfloor \frac{L_{2sn+m}}{L_{2s}^{n+1}} 
= \frac{L_{2s}^{k}( L_{2sk+m}L_{2s} + L_{2s(k-1)+m}) - (-1)^k(L_{2s}L_{m}+L_{m-2s})}
{(2L_{2s}^2 + 1)\big(L_{2s}^{2k} - (-1)^{k} L_{2s}^k L_{2sk} + 1\big)}.
\end{gather*}
\end{example}

Let $f^{+}(z,k)$ and $f^{-}(z,k)$ denote the functions in \eqref{ex1_floor_id1} and \eqref{ex1_floor_id2}, respectively, i.e.,
\begin{gather*}
f^{+}(z,k)=\sum_{n=0}^\infty \Big\lfloor \frac{n}{k} \Big\rfloor z^n = \frac{z^k}{(1-z)(1-z^k)},\\
f^{-}(z,k)=\sum_{n=0}^\infty (-1)^n \Big\lfloor \frac{n}{k} \Big\rfloor z^n = \frac{(-1)^k  z^k}{(1+z)(1+(-1)^{k+1}z^k)}.
\end{gather*}
\begin{lemma}
If $|z|<1$, then
\begin{gather}\label{eq.vjyhfag}
\sum_{n=1}^\infty \left\lfloor \frac{n}{2} \right\rfloor \frac{z^n}{n} = \frac{1}{4} \log \left| \frac{1 - z}{1 + z} \right| + \frac{z}{2(1-z)},\\
\label{eq.vtkra2n}
\sum_{n=1}^\infty (-1)^{n-1} \left\lfloor \frac{n}{2} \right\rfloor \frac{z^n}{n} = \frac{1}{4}\log \left| \frac{1 - z}{1 + z} \right| 
+ \frac{z}{2(1+z)}.
\end{gather}
\end{lemma}
\begin{proof}
To get the first identity integrate $f^{+}(z,2)/z$ and use $f^{+}(0,2) = 0$. The second follows from working with $f^{-}(z,2)$.
\end{proof}
\begin{example}
Setting $z=\sqrt 5/5$, $z=\sqrt 5/3$, $z=2/\sqrt 5$ and $z=3\sqrt 5/7$ in turn in \eqref{eq.vjyhfag}, \eqref{eq.vtkra2n} and adding we have
\begin{gather*}
\sum_{n=1}^\infty \frac{1}{5^n} \frac{n}{2n + 1} = \frac{5}{8} - \frac{\sqrt5}{2}\log \alpha, 
\qquad
\sum_{n=1}^\infty \left(\frac{5}{9}\right)^n\frac{n}{2n + 1} = \frac{9}{8} - \frac{3\sqrt{5}}{5}\log \alpha,\\
\sum_{n=1}^\infty \left(\frac{4}{5}\right)^n\frac{n}{2n + 1} = \frac{5}{2} - \frac{3\sqrt{5}}{4}\log \alpha,
\qquad
\sum_{n=1}^\infty \left(\frac{45}{49}\right)^n\frac{n}{2n + 1} = \frac{49}{8} - \frac{14\sqrt{5}}{15}\log \alpha.
\end{gather*}
\end{example}
\begin{remark} 
One can obtain all of the above formulas from \cite[Formula 5.2.4.]{Prud}: 
\begin{equation*}
\sum_{k=0}^\infty \frac{x^k}{2k+m} = \frac{1}{2x^{m/2}} \left( -\log (1-\sqrt{x}) + (-1)^{m+1} \log (1+\sqrt{x}) \right).
\end{equation*}
For instance, to get the first formula, insert $x=1/5$ and $m=1$ to get
\begin{equation*}
\sum_{k=0}^\infty \frac{\left(\frac{1}{5}\right)^k}{2k+1} = \frac{\sqrt5}{2}\log \left( \frac{\sqrt5+1}{\sqrt5-1}\right ) \quad\text{or}\quad 
\sum_{k=0}^\infty \frac{1}{(2k+1) 5^k} = {\sqrt 5} \log \alpha.
\end{equation*}
The formula follows from $\frac{2n}{2n+1}=1-\frac{1}{2n+1}$ and $\sum\limits_{k=1}^{\infty} \frac{1}{5^k}=\frac{1}{4}$.
\end{remark}
\begin{theorem}
If $m$ is any integer, then 
\begin{equation}\label{eq.vpc51jb}
{\sum_{n = 1}^\infty} \left\lfloor {\frac{n}{2}} \right\rfloor \frac{L_{n + m}}{2^n n} = - \frac{3\sqrt 5}{4}F_m  \log \alpha  
- \frac{L_m}{8} \log 5 + \frac{L_{m + 3}}{2}
\end{equation}
and
\begin{equation}\label{eq.njnootu}
{\sum_{n = 1}^\infty} \left\lfloor {\frac{n}{2}} \right\rfloor \frac{F_{n + m}}{2^n n} = - \frac{3}{4\sqrt 5 } L_m\log \alpha  
- \frac{F_m}{8} \log 5 + \frac{F_{m + 3}}{2}.
\end{equation}
\end{theorem}
\begin{proof} 
Use of $z=\alpha/2$ and $z=\beta/2$, in turn, in \eqref{eq.vjyhfag} produces
\begin{equation}\label{eq.kdzdqb7}
{\sum_{n = 1}^\infty} \left\lfloor {\frac{n}{2}} \right\rfloor  \frac{{\alpha ^{n + m} }}{{2^n n}} =  
- \frac{3}{4}\alpha^m \log \alpha - \frac{1}{8}\alpha^m \log 5 + \alpha^{m + 2} - \frac{{\alpha^m }}{2}
\end{equation}
and
\begin{equation}\label{eq.fv2fby8}
{\sum_{n = 1}^\infty}  \left\lfloor {\frac{n}{2}} \right\rfloor  \frac{{\beta ^{n + m} }}{{2^n n}} 
= \frac{3}{4}\beta^m \log \alpha - \frac{1}{8}\beta^m \log 5 + \beta^{m + 2} - \frac{{\beta^m }}{2},
\end{equation}
where $m$ is an arbitrary integer. Addition of \eqref{eq.kdzdqb7} and \eqref{eq.fv2fby8} yields \eqref{eq.vpc51jb} while their difference gives \eqref{eq.njnootu}.
\end{proof}

The alternating versions of \eqref{eq.vpc51jb} and \eqref{eq.njnootu} via \eqref{eq.vtkra2n} are stated next. As the proofs are similar we omit them.
\begin{theorem} 
If $m$ is any integer, then
\begin{equation*}
\sum_{n = 1}^\infty (-1)^n\left\lfloor \frac{n}{2} \right\rfloor  \frac{F_{n + m}}{2^n n} 
= \frac{F_{m}}{8} \log 5 + \frac{3}{4\sqrt 5}{L_m}\log \alpha - \frac{L_{m}}{10}
\end{equation*}
and
\begin{equation*}
\sum_{n = 1}^\infty (-1)^n \left\lfloor \frac{n}{2} \right\rfloor \frac{L_{n + m}}{2^n n} 
= \frac{L_m }{8}\log 5 + \frac{3\sqrt 5}{4}  F_m \log \alpha - \frac{ F_{m}}{2}.
\end{equation*}
\end{theorem}
\begin{lemma} 
If $m$ is a non-negative integer and $|z|<1$, then
\begin{gather}\label{eq.vv7hj30}
\sum_{n=1}^\infty \left\lfloor {\frac{n}{2}} \right\rfloor \binom {n}{m} z^{n - m} 
= - \frac{3}{4(1 - z)^{m + 1}} + \frac{(- 1)^m}{4(1 + z)^{m + 1}} + \frac{m + 1}{2(1 - z)^{m + 2}},\\
\label{eq.ax69bx3}
\sum_{n=1}^\infty (- 1)^{n - 1} \left\lfloor {\frac{n}{2}} \right\rfloor \binom {n}{m} z^{n - m} 
=  \frac{(- 1)^m3}{4(1 + z)^{m + 1}} - \frac{1}{4(1 - z)^{m + 1}} -  \frac{(- 1)^m(m+1)}{2(1 + z)^{m + 2}}.
\end{gather}
\end{lemma}
\begin{proof} 
Differentiate $f^{+}(z,2)$ and $f^{-}(z,2)$ $m$ times with respect to $z$.
\end{proof}
\begin{theorem}
If $m$ is a non-negative integer, then
\begin{gather*}
\sum_{n=1}^\infty \frac{n}{5^{n+1}} \binom {2n+1}{m} = \frac{(5m +1)F_{m + 1} + (5m-3)F_{m}}{2^{m + 4}},\\
\sum_{n=1}^\infty \frac{n}{5^{n + 1}} \binom {2n}{m} = \frac{(3m + 1)F_{m + 1}+(m+1)F_{m}}{2^{m + 4}}.
\end{gather*}
\end{theorem}
\begin{proof}
Set $z=\frac{1}{\sqrt 5}$ in \eqref{eq.vv7hj30} and \eqref{eq.ax69bx3}; add and subtract the resulting identities and simplify.
\end{proof}
\begin{theorem}
If $m$ is a non-negative integer, then
\begin{align*}
\sum_{n=1}^\infty n \binom{2n+1}{2m}\left(\frac{5}{9}\right)^n &= \frac{9}{16}\left(\frac{5}{4}\right)^m \big((6m+3)F_{4m+4} - 4F_{4m+2}\big),\\
\sum_{n=1}^\infty n \binom{2n+1}{2m+1}\left(\frac{5}{9}\right)^n &=\frac{9}{32}\left(\frac{5}{4}\right)^m \big((6m+6)L_{4m+6} - 4L_{4m+4}\big),\\
\sum_{n=1}^\infty n \binom{2n}{2m}\left(\frac{5}{9}\right)^n &=\frac{3}{16}\left(\frac{5}{4}\right)^m \big((6m+3)L_{4m+4} - 2L_{4m+2}\big)
\end{align*}
and
\begin{equation*}
\sum_{n=1}^\infty n \binom{2n}{2m+1}\left(\frac{5}{9}\right)^n =\frac{15}{16}\left(\frac{5}{4}\right)^m \big((3m+3)F_{4m+6} - F_{4m+4}\big).
\end{equation*}
\end{theorem}
\begin{proof}
Set $z=\sqrt5/3$ in \eqref{eq.vv7hj30} and \eqref{eq.ax69bx3}; add and subtract the resulting identities and simplify by considering the cases of even and odd $m$.
\end{proof}
\begin{theorem}  
For $p$ an integer and $m$ a non-negative integer, we have the identities
\begin{align*}
\sum_{n = 1}^\infty \left\lfloor \frac{n}{2} \right\rfloor \binom {n}{m} \frac{F_{n+p}}{2^{n-1}} & =  
4(m+1)F_{3m+p+4} - {3} F_{3m+p+2} + {(-1)^m}
\begin{cases}
5^{-{(m+2)}/2} L_{p-1}, & \text{\rm $m$ even;} \\[4pt] 5^{-{(m+1)}/2} F_{p-1}, & \text{\rm $m$ odd,} 
\end{cases} 
\end{align*}
and
\begin{align*}
\sum_{n = 1}^\infty \left\lfloor \frac{n}{2} \right\rfloor \binom {n}{m} \frac{L_{n+p}}{2^{n-1}} & =  
4(m+1)L_{3m+p+4} - {3} L_{3m+p+2} + {(-1)^m}
\begin{cases}
5^{-m/2} F_{p-1}, & \text{\rm $m$ even;} \\[4pt] 5^{-{(m+1)}/2} L_{p-1}, & \text{\rm $m$ odd.
}
\end{cases} 
\end{align*}
\end{theorem}
\begin{lemma}
For a positive integer $k$ and $|z|<1$,
\begin{gather}\label{eq.djwcsym} 
\sum_{n = 1}^\infty  {\left\lfloor {\frac{n}{k}} \right\rfloor \left( {\frac{{z^n }}{n} - \frac{{z^{n + 1} }}{{n + 1}}} \right)} 
= - \frac{1}{k} \log \left(1 - z^k\right)\!, \\
\label{eq.wn9acmr}
\sum_{n = 1}^\infty  {( - 1)^{n - 1} \left\lfloor {\frac{n}{k}} \right\rfloor \left( {\frac{{z^n }}{n} + \frac{{z^{n + 1} }}{{n + 1}}} \right)} = \frac{1}{k}\log \left(1 - (- 1)^k z^k \right).
\end{gather}
\end{lemma}
\begin{proof}
Write \eqref{ex1_floor_id1} as
\begin{equation*}
\sum_{n = 1}^\infty \left\lfloor {\frac{n}{k}} \right\rfloor (1 - z)z^{n - 1} = \frac{{z^{k - 1} }}{{1 - z^k }}
\end{equation*}
and integrate both sides with respect to $z$ to obtain \eqref{eq.djwcsym}. The second identity follows from \eqref{eq.djwcsym} by replacing $z$ with $-z$.
\end{proof}
\begin{theorem}\label{th8} 
If $k$ is an integer with $k\ge 1$, $r$ is any integer and $m$ is an even integer, then
\begin{equation}\label{eq.h0nbyst}
\begin{split}
&\qquad\qquad\qquad\sum_{n = 1}^\infty \frac{\left\lfloor {\frac{n}{k}} \right\rfloor}{L_m^{n+1}} \left( {\frac{{L_{mn + r}L_m }}{{n}} - \frac{{L_{m(n + 1) + r} }}{n + 1}} \right)\\
& = -\frac{\beta^r}k\log \left( {L_m^{2k}  - L_m^k L_{mk}  + 1} \right)- \frac{\sqrt 5}{k}F_r  \log \left( {L_m^k  - \alpha ^{mk} } \right) + L_r \log L_m,
\end{split}
\end{equation}
\begin{equation}\label{eq.bxdefgi}
\begin{split}
&\qquad\qquad\qquad\sum_{n = 1}^\infty \frac{\left\lfloor {\frac{n}{k}} \right\rfloor}{L_m^{n+1}} \left( {\frac{F_{mn + r}L_m }{n} - \frac{F_{m(n + 1) + r}}{n + 1 }} \right)\\
& = \frac{\beta^r}{\sqrt 5\,k}\log \left( {L_m^{2k}  - L_m^k L_{mk}  + 1} \right) - \frac{L_r}{ \sqrt 5\,k} \log \left( {L_m^k  - \alpha ^{mk} } \right) + F_r\log L_m.
\end{split}
\end{equation}
\end{theorem}
\begin{proof}
Choosing $z=\alpha^m/L_m$ and $z=\beta^m/L_m$ in \eqref{eq.djwcsym} produces
\[
\sum_{n = 1}^\infty  {\left\lfloor {\frac{n}{k}} \right\rfloor \left( {\frac{{\alpha ^{mn + r} }}{{nL_m^n }} - \frac{{\alpha ^{m(n + 1) + r} }}{{(n + 1)L_m^{n + 1} }}} \right)}  =  - \frac{{\alpha ^r }}{k}\log \left( {\frac{{L_m^k  - \alpha ^{mk} }}{{L_m^k }}} \right)
\] 
and
\[
\sum_{n = 1}^\infty  {\left\lfloor {\frac{n}{k}} \right\rfloor \left( {\frac{{\beta ^{mn + r} }}{{nL_m^n }} - \frac{{\beta ^{m(n + 1) + r} }}{{(n + 1)L_m^{n + 1} }}} \right)}  =  - \frac{{\beta ^r }}{k}\log \left( {\frac{{L_m^k  - \beta ^{mk} }}{{L_m^k }}} \right);
\] 
from which \eqref{eq.h0nbyst} and \eqref{eq.bxdefgi} follow.
\end{proof}
\begin{theorem} If $k$ is an integer with $k\ge 1$, $r$ is any integer and $m$ is an odd integer, then
\begin{align*}
&\qquad\qquad\qquad\sum_{n = 1}^\infty \frac{\left\lfloor {\frac{n}{k}} \right\rfloor}{\sqrt{5^n} F_m^{n+1}} \!\left(   \frac{\sqrt5 L_{mn + r}F_m}{n} - \frac{L_{m(n + 1) + r}}{n + 1}\right)\! \\ 
&=\! -\frac{\sqrt5 \alpha^r}{k} \log\! \big( {5^kF_m^{2k}  - \sqrt{5^k}F_m^k L_{mk}  + (-1)^k} \big) + \frac{5 F_r}{k}  \log \big( {\sqrt{5^k} F_m^k  - \beta ^{mk} } \big) + \sqrt5 L_r\log (\sqrt5 F_m),\\
&\qquad\qquad\qquad\sum_{n = 1}^\infty \frac{\left\lfloor {\frac{n}{k}} \right\rfloor}{\sqrt{5^n} F_m^{n+1}}\!\left(   \frac{\sqrt5 F_{mn + r}F_m}{n} - \frac{F_{m(n + 1) + r}}{n + 1}\right)\! \\
&= - \frac{\alpha^r}{ k} \log \left( {5^kF_m^{2k}  - \sqrt{5^k}F_m^k L_{mk}  + (-1)^k} \right) + \frac{ L_r}{ k}  \log \big( {\sqrt{5^k} F_m^k  - \beta ^{mk} } \big) + \sqrt5 F_r\log (\sqrt5 F_m).
\end{align*}
\end{theorem}
\begin{proof}
Set $z=\alpha^m/(\sqrt5F_m)$ and $z=\beta^m/(\sqrt5F_m)$, in turn, in \eqref{eq.djwcsym} and proceed as in the proof of Theorem \ref{th8}.
\end{proof}
The alternating versions of \eqref{eq.h0nbyst} and \eqref{eq.bxdefgi}, obtained from \eqref{eq.wn9acmr}, are left as an exercise.

\medskip
In the next Lemma we present a generalization of \eqref{eq.djwcsym} and \eqref{eq.wn9acmr}.
\begin{lemma}
For an integer $k\ge 1$, any integer $m$ and $|z|\leq1$,
\begin{gather}\label{eq.x8nwnxn}
\sum_{n = 1}^\infty  {\left\lfloor {\frac{n}{k}} \right\rfloor \left( {\frac{{z^n }}{{n^m }} - \frac{{z^{n + 1} }}{{(n + 1)^m }}} \right)}  = \frac{1}{{k^m }}\Li_m (z^k ),\\
\sum_{n = 1}^\infty {( - 1)^{n - 1} \left\lfloor {\frac{n}{k}} \right\rfloor \left( {\frac{{z^n }}{{n^m }} + \frac{{z^{n + 1} }}{{(n + 1)^m }}} \right)} = - \frac{1}{{k^m }}\Li_m \big((- 1)^k z^k \big),\nonumber
\end{gather}
where $\Li_m(x)=\sum\limits_{k=1}^{\infty}\frac{x^k}{k^m}$ is a polylogarithm.
\end{lemma}
\begin{proof}
We proof \eqref{eq.x8nwnxn} by induction on $m$ for a fixed $k$. The identity is valid for $m=1$, corresponding to \eqref{eq.djwcsym}. Assume that it is valid for an integer $m=r$. We therefore have the hypothesis:
\[
P_r:\quad\sum_{n = 1}^\infty  {\left\lfloor {\frac{n}{k}} \right\rfloor \left( {\frac{{z^n }}{{n^r }} - \frac{{z^{n + 1} }}{{(n + 1)^r }}} \right)}  = \frac{1}{{k^r }}\Li_r (z^k ).
\]
We wish to prove that $P_r\implies P_{r + 1}$.

The $P_r$ identity is true for $z=0$. Dividing through by $z\ne0$ gives
\[
\sum_{n = 1}^\infty  {\left\lfloor {\frac{n}{k}} \right\rfloor \left( {\frac{{z^{n - 1} }}{{n^r }} - \frac{{z^n }}{{(n + 1)^r }}} \right)}  =   \frac{1}{{k^r }}\frac{{\Li_r (z^k )}}{z},
\]
which, upon integration with respect to $z$ from $0$ to $z$ yields
\[
P_{r + 1}:\quad\sum_{n = 1}^\infty  {\left\lfloor {\frac{n}{k}} \right\rfloor \left( {\frac{{z^n }}{{n^{r + 1} }} - \frac{{z^{n + 1} }}{{(n + 1)^{r + 1} }}} \right)}  = \frac{1}{{k^{r + 1} }}\Li_{r + 1} (z^k ),
\]
since
\begin{equation*}
\int_0^z \frac{{\Li_r (x^k )}}{x} dx = \int_0^z \sum_{n = 1}^\infty \frac{{x^{kn - 1} }}{{n^r }} dx
= \sum_{n = 1}^\infty \frac{1}{{n^r }} \int_0^z x^{kn - 1} dx = \frac{1}{k} \sum_{n = 1}^\infty \frac{{z^{kn} }}{{n^{r + 1} }} 
= \frac{{\Li_{r + 1} (z^k )}}{k}.
\end{equation*}
Thus, $P_r\implies P_{r + 1}$.
\end{proof}
\begin{remark}
Since $\Li_m(e^{2\pi i})=\zeta(m)$ from \eqref{eq.x8nwnxn} we have a representation of the Riemann zeta function as follows:
$$
\zeta(m) = k^m \sum_{n=1}^{\infty} \left\lfloor {\frac{n}{k}} \right\rfloor \frac{(n+1)^m - n^m}{n^m(n+1)^m}.	
$$
\end{remark}
\begin{lemma}
Let $H_j$ be the $j$-th harmonic number. Let $k$ be a positive integer. If $|z|<1$, then
\begin{equation}\label{eq.ydbk6v1}
\begin{split}
\sum_{n = 0}^\infty {\frac{{H_{\left\lfloor {n/k} \right\rfloor } z^{n + k} }}{{\left( {\left\lfloor {n/k} \right\rfloor  + 1} \right)^2 }}}  =\frac{1 - z^k }{1 - z}&\Bigl( \frac{k}{2}\log z\log ^2( {1 - z^k }) \\
&\quad { + \log( {1 - z^k } )\Li_2 ( {1 - z^k } ) - \Li_3 ( {1 - z^k } ) + \zeta (3)} \Bigr).
\end{split}
\end{equation}
\end{lemma}
\begin{proof}
It is known that \cite[Formula (12), p.~303]{lewin81}:
\[
\sum_{n = 0}^\infty  {\frac{{H_n z^{n + 1} }}{{(n + 1)^2 }}} = {\frac{1}{2}} \log z\log ^2 (1 - z) + \log (1 - z)\Li_2 (1 - z) - \Li_3 (1 - z) + \zeta (3).
\]
Thus with $a_n=H_n/(n+1)^2$ in \eqref{gf_floor_id1}, the result follows.
\end{proof}
\begin{remark}
Since 
\begin{equation*}
\Li_2\Big(\frac12\Big)=\frac{\pi^2}{12}-\frac{\log^22}{2},\qquad \Li_3\Big(\frac12\Big)=\frac{\log^32}{6}-\frac{\pi^2\log2}{12}+\frac{7}{8}\zeta(3),
\end{equation*}
from \eqref{eq.ydbk6v1} we have the following representation of $\zeta(3)$:
\begin{equation*}
\zeta(3) = 4(2-\sqrt2) \sum_{n=1}^{\infty}
\frac{H_{\left\lfloor {\frac{n}{2}}\right\rfloor}}{2^{\frac{n}{2}}\left\lfloor {\frac{n+2}{2}} \right\rfloor^2} + \frac{4}{3}\log^3 2.	
\end{equation*}
\end{remark}
\begin{lemma}\label{lem6}
Let $k$ be a positive even integer and $|z|< 1$. Let $(a_n)_{n\geq 0}$ be an arbitrary sequence. Then
\begin{equation}\label{eq.tu4irvm}
(1 - z)\sum_{n = 0}^\infty a_{\left\lfloor {n/k} \right\rfloor } z^n 
= (1 + z)\sum_{n = 0}^\infty (- 1)^n a_{\left\lfloor {n/k} \right\rfloor } z^n.
\end{equation}
In particular, for all convergent sequences $(a_n)_{n\geq 0}$
\begin{equation*}
\sum_{n = 0}^\infty (-1)^n a_{\left\lfloor {n/k} \right\rfloor } = 0, \quad \mbox{$k$ even}.
\end{equation*}
\end{lemma}
\begin{proof}
Let $h(z,k)=(1-z)F^+(z,k)$. Clearly, from \eqref{gf_floor_id1}, $h(-z,k)=h(z,k)$ if $k$ is an even integer; and hence~\eqref{eq.tu4irvm}.
\end{proof}
\begin{example} 
Lemma \ref{lem6} yields for $k\geq1$:
\begin{gather*}
\sum_{n = 1}^\infty \frac{1-3(-1)^n	}{2^n} F_{\lfloor	\frac{n}{2k}\rfloor} = 0, \qquad 
\sum_{n = 1}^\infty \frac{1-3(-1)^n	}{2^n} L_{\lfloor \frac{n}{2k}\rfloor} = 0,\\ 
\sum_{n = 1}^\infty \frac{1-2(-1)^n	}{3^n} \left\lfloor \frac{n}{2k}\right\rfloor F_{\lfloor \frac{n}{2k}\rfloor} = 0, \qquad 
\sum_{n = 1}^\infty \frac{1-2(-1)^n	}{3^n} \left\lfloor \frac{n}{2k}\right\rfloor L_{\lfloor \frac{n}{2k}\rfloor} = 0,\\ 
\sum_{n = 1}^\infty \frac{3-5(-1)^n	}{4^n} F_{k\lfloor \frac{n}{2k}\rfloor} = 0, \qquad 
\sum_{n = 1}^\infty \frac{3-5(-1)^n	}{4^n} L_{k\lfloor \frac{n}{2k}\rfloor} = 0. 
\end{gather*}
\end{example}
\begin{theorem}\label{th10}
If $k$ is a positive even integer, $m$ any integer, $p\ge2$ and $(a_n)_{n\geq 0}$ an arbitrary sequence, then
\begin{align}\label{eq.gjrt7dj}
\sum_{n = 0}^\infty  a_{\left\lfloor {n/k} \right\rfloor } \frac{p{L_{n + m } - L_{n+m+1} }}{p^{n}}  
&= \sum_{n = 0}^\infty  (-1)^n a_{\left\lfloor {n/k} \right\rfloor} \frac{pL_{n + m } + L_{n+m+1} }{p^{n}},\\
\label{eq.iorv3sv}
\sum_{n = 0}^\infty  a_{\left\lfloor {n/k} \right\rfloor } \frac{p{F_{n + m } - F_{n+m+1} }}{p^{n}}  
&= \sum_{n = 0}^\infty  (-1)^n a_{\left\lfloor {n/k} \right\rfloor} \frac{pF_{n + m } + F_{n+m+1} }{p^{n}}.
\end{align}
\end{theorem}
\begin{proof}
Setting $z=\alpha/p$ and $z=\beta/p$, respectively, in \eqref{eq.tu4irvm} gives
\[
(p-\alpha)\sum_{n = 0}^\infty {a_{\left\lfloor {n/k} \right\rfloor } \frac{{\alpha ^{n + m} }}{p^n}}  
= (p+\alpha) \sum_{n = 0}^\infty  {( - 1)^n a_{\left\lfloor {n/k} \right\rfloor } \frac{{\alpha ^{n + m}  }}{p^n }}
\]
and
\[
(p-\beta)\sum_{n = 0}^\infty {a_{\left\lfloor {n/k} \right\rfloor } \frac{{\beta ^{n + m} }}{p^n}}  
= (p+\beta) \sum_{n = 0}^\infty  {( - 1)^n a_{\left\lfloor {n/k} \right\rfloor } \frac{{\beta^{n + m}  }}{p^n }}
\]
where $m$ is an arbitrary integer. Respective addition and subtraction of these two identities produce \eqref{eq.gjrt7dj} 
and \eqref{eq.iorv3sv}.
\end{proof}
\begin{example} 
For $p=2$ and $p=3$ in  Theorem \ref{th10} we have:
\begin{gather*}
\sum_{n = 0}^\infty \frac{a_{\left\lfloor {n/k} \right\rfloor }}{2^n}\big(L_{n + m - 2} - 5(-1)^nF_{n + m + 1}\big) = 0,\\
\sum_{n = 0}^\infty \frac {a_{\left\lfloor {n/k} \right\rfloor }}{2^n}\big( F_{n + m - 2} - ( - 1)^n L_{n + m + 1}\big) = 0,
\end{gather*}
\begin{gather*}
\sum_{n = 0}^\infty  \frac{a_{\left\lfloor {n/k} \right\rfloor }}{3^n}\big(L_{n + m - 1} - ( - 1)^n (F_{n+m} + L_{n + m + 1})\big) = 0,\\
\sum_{n = 0}^\infty  \frac{a_{\left\lfloor {n/k} \right\rfloor }}{3^n}\big(5F_{n + m - 1} - ( - 1)^n (L_{n+m} + 5F_{n + m + 1})\big) = 0.
\end{gather*}
\end{example}
\begin{theorem}
If $k$ is a positive even integer, $m$ is an even integer, $r$ any integer and $(a_n)_{n\geq 0}$ an arbitrary sequence, then
\begin{align*}
\sum_{n = 1}^\infty &\frac{a_{\left\lfloor {n/k} \right\rfloor }}{L_m^{n}}\big(L_{mn + r - m} -(-1)^n(L_{mn + r} L_m + L_{m(n + 1) + r}) \big)=0,\\
\sum_{n = 1}^\infty &\frac{a_{\left\lfloor {n/k} \right\rfloor }}{L_m^{n}} \big(F_{mn + r - m}- ( - 1)^n ( F_{mn + r} L_m + F_{m(n + 1) + r})\big)=0.
\end{align*}
\end{theorem}
\begin{proof}
Set $z=\alpha^m/L_m$ and $z=\beta^m/L_m$, in turn, in \eqref{eq.tu4irvm} and proceed as in the proof of Theorem \ref{thm_2}.
\end{proof}
\begin{theorem}
If $k$ is a positive even integer, $m$ is an odd integer, $r$ is any integer and $(a_n)_{n\geq 0}$ an arbitrary sequence, then
\begin{align*}
\sum_{n = 1}^\infty \frac{ a_{\left\lfloor {2n/k} \right\rfloor } F_{2mn+r+m} - a_{\left\lfloor {(2n-1)/k} \right\rfloor }
F_m\big(L_{2m(n-1) + r} + L_{2mn+r}\big)}{{5^nF_m^{2n}}} &= 0,\\
\sum_{n = 1}^\infty \frac{ a_{\left\lfloor {2n/k} \right\rfloor } L_{2mn+r+m}- 5  a_{\left\lfloor {(2n-1)/k} \right\rfloor }
F_m(F_{2m(n-1) + r} + F_{2mn+r})}{{5^n F_m^{2n}}} &= 0.
\end{align*}
\end{theorem}
\begin{proof}
Set $z=\alpha^m/(\sqrt5F_m)$ and $z=\beta^m/(\sqrt5F_m)$, in turn, in \eqref{eq.tu4irvm} and proceed as in the proof of Theorem \ref{thm_2}.
\end{proof}

\section{The binomial transform of $(a_{\lfloor n/k \rfloor})_{n\geq 0}$}

We begin with recalling some basic facts. Let $(a_n)_{n\geq 0}$ be an arbitrary sequence of numbers
with ordinary generating function $F(z)$. Let further $(s_n)_{n\geq 0}$ be the generalized binomial transform of 
$(a_n)_{n\geq 0}$, i.e., 
\begin{equation*}
s_n = \sum_{j=0}^n \binom {n}{j} b^{n-j} c^j a_j 
\end{equation*}
with $b$ and $c$ being nonzero real numbers. Then the ordinary generating function of the sequence $(s_n)_{n\geq 0}$ is given 
by \cite{Boyadzhiev,Prod}
\begin{equation*}
S(z) = \frac{1}{1-bz} F\left ( \frac{cz}{1-bz}\right ).
\end{equation*}
Therefore, the ordinary generating function of the sequence
\begin{equation*}
s_{n}^* = \sum_{j=0}^n \binom {n}{j} b^{n-j} c^j a_{\lfloor j/k \rfloor}
\end{equation*}
equals
\begin{align}\label{bin_gen_fkt}
S^* (z) & = \frac{1}{1-bz} \frac{ 1-\left (\frac{cz}{1-bz}\right )^k }{1-\frac{cz}{1-bz}}\, 
F\!\left ( \Bigl (\frac{cz}{1-bz}\Bigr )^k  \right ) = \frac{ 1-\left (\frac{cz}{1-bz}\right )^k }{1-(b+c)z} \,F\!\left ( \Bigl(\frac{cz}{1-bz}\Bigr )^k  \right ).
\end{align}

In general, such a function will be a complicated expression. Focusing on the sequence $a_{n}=n$ we have
\begin{equation*}
S^* (z,k) = \frac{ 1-\left (\frac{cz}{1-bz}\right )^k }{1-(b+c)z} 
\frac{\left ( \frac{cz}{1-bz}\right )^k}{\left ( 1 - \left ( \frac{cz}{1-bz}\right )^k \right )^2} 
= \frac{(cz)^k}{(1-(b+c)z)((1-bz)^k - (cz)^k)}.
\end{equation*}

The complex function $P(z) = (1-bz)^k - (cz)^k$ is a polynomial of degree $k$ for $b\neq \pm c$ and we have
$$P(z) = (z-r_1)(z-r_2)\cdots (z-r_k),$$ 
where $r_i\in\mathbb{C}$, $i=1,\ldots,k,$ are the roots of $P(z)$. The partial fraction decomposition yields
\begin{equation*}
\frac{1}{P(z)} = \sum_{i=1}^k\frac{1}{z-r_i}  \prod_{j=1,\\ \,j\neq i}^k \frac{1}{r_j - r_i}.
\end{equation*}
This shows that (at least theoretically) the function $S^*(z,k)$ can be expressed as a convolution. But even for small values of $k$ this will be a challenging issue. Our first result is therefore concerned with 
analyzing the case $k=2$.
\begin{theorem}\label{thm_bin}
For any numbers $b$ and $c$ with $c\ne0$ and $n\geq 1$, we have
\begin{equation}\label{main_bin_id}
{\sum_{j=1}^n} \binom {n}{j}\Big\lfloor \frac{j}{2} \Big\rfloor b^{n-j} c^j  =
\frac{cn}{2}(b+c)^{n-1} - \frac{ (b+c)^n - (b-c)^n }{4}.
\end{equation}
If $b=c$, then this identity becomes
\begin{equation*}
{\sum_{j=1}^n} \binom {n}{j} \Big\lfloor \frac{j}{2} \Big\rfloor = 2^{n-2} (n-1).
\end{equation*}
\end{theorem}
\begin{proof}
If $k=2$ we simplify $S^*(z,2)$ and get
\begin{equation*}
S^*(z,2) = \frac{(cz)^2}{(1-(b+c)z)(1 - 2bz + (b^2-c^2)z^2)}.
\end{equation*}
Obviously, near zero, there exists the power series
\begin{equation*}
\frac{1}{1-(b+c)z} = \sum_{n=0}^\infty (b+c)^n z^n.
\end{equation*}
Also, from the partial fraction decomposition
\begin{equation*}
\frac{1}{1 - 2bz + (b^2-c^2)z^2} = \frac{b+c}{2c(1-(b+c)z)} - \frac{b-c}{2c(1-(b-c)z)}
\end{equation*}
we get (again near zero)
\begin{equation*}
\frac{1}{1 - 2bz + (b^2-c^2)z^2} = \frac{1}{2c} \sum_{n=0}^\infty \big((b+c)^{n+1} - (b-c)^{n+1} \big )z^n.
\end{equation*}
Hence, applying Cauchy's product formula for power series   
\begin{equation*}
S^*(z,2) = \frac{c}{2} \sum_{n=2}^\infty \sum_{j=0}^{n-2} (b+c)^j \big((b+c)^{n-1-j} - (b-c)^{n-1-j} \big) z^n
\end{equation*}
or, for $n\geq 2$,
\begin{align*}
{\sum_{j=1}^n} \binom {n}{j}  \Big\lfloor \frac{j}{2} \Big\rfloor b^{n-j} c^j& = 
\frac{c}{2} \sum_{j=0}^{n-2} (b+c)^j \left ((b+c)^{n-1-j} - (b-c)^{n-1-j} \right ) \\
& =  \frac{c}{2} (b+c)^{n-1} (n-1) - \frac{c}{2} (b-c)^{n-1} \sum_{j=0}^{n-2} \left (\frac{b+c}{b-c} \right )^j.
\end{align*}
The statement now follows from simplifying making use of the geometric series.
\end{proof}
\begin{corollary}
For $n\geq 2$ we have
\begin{align*}
{\sum_{j=1}^n} (-1)^{j} \binom {n}{j}  \Big\lfloor \frac{j}{2} \Big\rfloor = 2^{n-2},\qquad
{\sum_{j=1}^n} \binom {n}{j} 2^{j} \Big\lfloor \frac{j}{2} \Big\rfloor = n 3^{n-1} - \frac{ 3^n - (-1)^{n}}{4}.
\end{align*}
\end{corollary}

We proceed with some new binomial sums involving Fibonacci and Lucas numbers.
\begin{theorem}\label{thm_bin_Fib}
For $m$ an integer, we have
\begin{equation}\label{bin_Fib1}
{\sum_{j=1}^n} \binom {n}{j}  \Big\lfloor \frac{j}{2} \Big\rfloor F_{j+m}
= \frac{n}{2} F_{2n+m-1} - \frac{1}{4} \big ( F_{2n+m} + (-1)^m F_{n-m} \big)
\end{equation}
and
\begin{equation}\label{bin_Luc1}
{\sum_{j=1}^n} \binom {n}{j} \Big\lfloor \frac{j}{2} \Big\rfloor L_{j+m} 
= \frac{n}{2} L_{2n+m-1} - \frac{1}{4} \big ( L_{2n+m} - (-1)^m L_{n-m} \big).
\end{equation}
\end{theorem}
\begin{proof}
Work with $(b,c)=(1,\alpha)$ and $(b,c)=(1,\beta)$ in Theorem \ref{thm_bin}. When simplifying use $F_{-n}=(-1)^{n-1} F_n$
and $L_{-n}=(-1)^{n} L_n$, respectively.
\end{proof}

Identities \eqref{bin_Fib1} and \eqref{bin_Luc1} should be compared to the classical results \cite{CarFer,Vinson} 
\begin{equation*} 
\sum_{j=0}^n \binom {n}{j} F_{j+m} = F_{2n+m} \quad\mbox{and}\quad \sum_{j=0}^n \binom {n}{j} L_{j+m} = L_{2n+m}.
\end{equation*}
\begin{theorem}
For $m$ an integer and $p$ an odd integer, we have
\begin{equation}\label{bin_Fib2} 
\begin{split}
\sum_{j=1}^n \binom {n}{j}  \Big\lfloor \frac{j}{2} \Big\rfloor &F_{2pj+m} 
 \\
&=\begin{cases}
	\frac{n\sqrt{5^{n-1}}}{2}F_p^{n-1}F_{p(n+1)+m} - \frac{\sqrt{5^{n-1}}}{4}F_p^nL_{pn+m} -\frac{1}{4} L_p^nF_{pn+m}, & \text{\rm $n$ odd;} \\[6pt]
	\frac{n\sqrt{5^{n-2}}}{2} F_p^{n-1} L_{p(n+1)+m} - \frac{\sqrt{5^{n}}}{4} F_p^{n} F_{pn+m} + \frac{1}{4} L_p^nF_{pn+m}, & \text{\rm $n$ even,} 
\end{cases}
\end{split} 
\end{equation}
and
\begin{equation}\label{bin_Luc2}
\begin{split}
\sum_{j=1}^n \binom {n}{j} \Big\lfloor \frac{j}{2} \Big\rfloor& L_{2pj+m}\\ 
&= \begin{cases}
	\frac{n\sqrt{5^{n-1}}}{2}F_p^{n-1} L_{p(n+1)+m} - \frac{\sqrt{5^{n+1}}}{4}F_p^n F_{pn+m} -\frac{1}{4} L_p^nL_{pn+m}, & \text{\rm $n$ odd;} \\[6pt]
	\frac{n\sqrt{5^n}}{2} F_p^{n-1} F_{p(n+1)+m} - \frac{\sqrt{5^n}}{4} F_p^{n} L_{pn+m} + \frac{1}{4} L_p^nL_{pn+m}, & \text{\rm $n$ even,} 
\end{cases} 
\end{split}
\end{equation}
\end{theorem}
\begin{proof}
Work with $(b,c)=(1,\alpha^p)$ and $(b,c)=(1,\beta^p)$ in Theorem \ref{thm_bin} using, for odd $p$,
$$1+\alpha^{2p}=\sqrt5 F_p\alpha^p,\quad 1+\beta^{2p}=-\sqrt5 F_p\beta^p,\quad 1-\alpha^{2p}=-L_p\alpha^p,\quad 1-\beta^{2p}=-L_p\beta^p.$$  
\end{proof}
\begin{theorem}
For $m$ an integer and $p$ an even integer, we have
\begin{equation*}
\sum_{j=0}^n \binom {n}{j} \Big\lfloor \frac{j}{2} \Big\rfloor F_{2pj+m}= \begin{cases}
	\frac{n}{2}L_p^{n-1}F_{p(n+1)+m} - \frac{1}{4} L_p^nF_{pn+m} -\frac{\sqrt{5^{n-1}}}{4} F_p^nL_{pn+m}, & \text{\rm $n$ odd;} \\[6pt]
	\frac{n}{2} L_p^{n-1} F_{p(n+1)+m} - \frac{1}{4} L_p^{n} F_{pn+m} + \frac{\sqrt{5^n}}{4} F_p^nF_{pn+m}, & \text{\rm $n$ even,} 
\end{cases} 
\end{equation*}
\begin{equation*}
\sum_{j=0}^n \binom {n}{j} \Big\lfloor \frac{j}{2} \Big\rfloor L_{2pj+m} 
= \begin{cases}
	\frac{n}{2} L_p^{n-1} L_{p(n+1)+m} - \frac{1}{4}L^n_p L_{pn+m} - \frac{\sqrt{5^{n+1}}}{4} F_p^n F_{pn+m}, & \text{\rm $n$ odd;} \\[6pt] 
	\frac{n}{2} L_p^{n-1} L_{p(n+1)+m} - \frac{1}{4} L_p^n L_{pn+m} + \frac{\sqrt{5^n}}{4} F_p^n L_{pn+m}, & \text{\rm $n$ even.} 
\end{cases} 
\end{equation*}
\end{theorem}
\begin{proof}
Work with $(b,c)=(1,\alpha^p)$ and $(b,c)=(1,\beta^p)$ in Theorem \ref{thm_bin} using, for even $p$,
$$ 1+\alpha^{2p}=L_p\alpha^p,\quad 1+\beta^{2p}=L_p\beta^p, \quad 1-\alpha^{2p}=-\sqrt5 F_p\alpha^p, \quad 1-\beta^{2p}=\sqrt5 F_p\beta^p.$$  
\end{proof}

The classical counterparts of identities \eqref{bin_Fib2} and \eqref{bin_Luc2} are \cite{CarFer}
\begin{equation*}
\sum_{j=0}^n \binom {n}{j} F_{2j+m} = \begin{cases}
5^{(n-1)/2} L_{n+m}, & \text{\rm $n$ odd;} \\[4pt] 
5^{n/2} F_{n+m}, & \text{\rm $n$ even,} 
\end{cases}  
\end{equation*}
as well as
\begin{equation*}
\sum_{j=0}^n \binom {n}{j} L_{2j+m} = \begin{cases}
5^{(n+1)/2} F_{n+m}, & \text{\rm $n$ odd;} \\[4pt]
5^{n/2} L_{n+m}, & \text{\rm $n$ even.} 
\end{cases} 
\end{equation*}
\begin{theorem}
For $m$ an integer, we have
\begin{align*}
\sum_{j=1}^n\binom {n}{j}  \Big\lfloor \frac{j}{2} \Big\rfloor  F_{3j+m}
&= 2^{n-2} \big ( n F_{2n+m+1} - F_{2n+m} + (-1)^n F_{n+m} \big),\\
\sum_{j=1}^n \binom {n}{j}  \Big\lfloor \frac{j}{2} \Big\rfloor L_{3j+m} 
&= 2^{n-2} \big ( n L_{2n+m+1} - L_{2n+m} + (-1)^n L_{n+m} \big ).
\end{align*}
\end{theorem}
\begin{proof}
Work with $(b,c)=(1,\alpha^3)$ and $(b,c)=(1,\beta^3)$ in Theorem \ref{thm_bin}. When simplifying use 
$$1 - \alpha^3 = - 2\alpha,\quad 1 - \beta^3 = - 2\beta, \quad 1+ \alpha^3 = 2\alpha^2,\quad 1 + \beta^3 = 2\beta^2.$$
\end{proof}
\begin{theorem}
For any integer $m$, we have
\begin{equation*}\label{bin_Fib4}
\sum_{j=1}^n \binom {n}{j}\Big\lfloor \frac{j}{2} \Big\rfloor 2^j F_{j+m}  =
\begin{cases}
n F_{3n+m-2} - \frac{1}{4} F_{3n+m} + \frac{5^{n/2}}{4}  F_{m}, & \text{\rm $n$ even;} \\[6pt]
n F_{3n+m-2} - \frac{1}{4} F_{3n+m} - \frac{5^{(n-1)/2}}{4}  L_{m}, & \text{\rm $n$ odd,} 
\end{cases} 
\end{equation*}
and
\begin{equation*}
\sum_{j=1}^n \binom {n}{j} \Big\lfloor \frac{j}{2} \Big\rfloor 2^j L_{j+m}  =
\begin{cases}
n L_{3n+m-2} - \frac{1}{4} L_{3n+m} + \frac{5^{n/2}}{4}  L_{m}, & \text{\rm $n$ even;} \\[6pt] 
n L_{3n+m-2} - \frac{1}{4} L_{3n+m} - \frac{5^{(n+1)/2}}{4}  F_{m}, & \text{\rm $n$ odd.} 
\end{cases} 
\end{equation*}
\end{theorem}
\begin{proof}
Work with $(b,c)=(1,2\alpha)$ and $(b,c)=(1,2\beta)$ in Theorem \ref{thm_bin}. When simplifying use 
$$1 - 2\alpha = - \sqrt{5},\quad 1 - 2\beta = \sqrt{5}, \quad 1 + 2\alpha = \alpha^3, \quad 1 + 2\beta = \beta^3.$$
\end{proof}

The classical counterparts of these identities are \cite{CarFer}
\begin{equation*}
\sum_{j=0}^n \binom {n}{j} 2^j F_{j+m} = F_{3n+m} \qquad\mbox{and}\qquad \sum_{j=0}^n \binom {n}{j} 2^j L_{j+m} = L_{3n+m}.
\end{equation*}

Although it is possible to state some more binomial identities we conclude with the following sums.
\begin{theorem}\label{thm5_bin_Fib}
For any integers $m$ and $q$, we have
\begin{align*}
\sum_{j=0}^n (-1)^{q(n-j)} &\binom {n}{j}   \Big\lfloor \frac{j}{2} \Big\rfloor F_{2qj+m}\\
& =
\begin{cases}
\frac{n}{2} F_{q(n+1)+m} L_q^{n-1} - \frac{1}{4} F_{qn+m} L_q^n +\frac{5^{n/2}}{4}  F_{qn+m} F_q^n, & \text{\rm $n$ even;} \\[6pt]
\frac{n}{2} F_{q(n+1)+m} L_q^{n-1} - \frac{1}{4} F_{qn+m} L_q^n - \frac{5^{(n-1)/2}}{4}  L_{qn+m} F_q^n, & \text{\rm $n$ odd,} 
\end{cases} \\
\sum_{j=0}^n (-1)^{q(n-j)} &\binom {n}{j}  \Big\lfloor \frac{j}{2} \Big\rfloor L_{2qj+m} \nonumber\\
& =
\begin{cases}
\frac{n}{2} L_{q(n+1)+m} L_q^{n-1} - \frac{1}{4} L_{qn+m} L_q^n + \frac{5^{n/2}}{4}  L_{qn+m} F_q^n, & \text{\rm $n$ even;} \\[6pt]
\frac{n}{2} L_{q(n+1)+m} L_q^{n-1} - \frac{1}{4} L_{qn+m} L_q^n - \frac{5^{(n+1)/2}}{4}  F_{qn+m} F_q^n, & \text{\rm $n$ odd.} 
\end{cases} 
\end{align*}
\end{theorem}
\begin{proof} Work with $(b,c)=((-1)^q,\alpha^{2q})$ and $(b,c)=((-1)^q,\beta^{2q})$ in Theorem \ref{thm_bin}, respectively. When simplifying use 
$$(- 1)^q + \alpha^{2q} = \alpha^q L_q,\,\, (- 1)^q - \alpha^{2q} = - \alpha^q F_q \sqrt 5,\,\, (- 1)^q + \beta^{2q} = \beta^q L_q,\,\, (- 1)^q - \beta^{2q} = \beta^q F_q \sqrt 5.$$
\end{proof}
\begin{corollary} 
Let $n$, $r$ and $s$ be non-negative integers such that $n\ge r + s$. Let $b$ and $c$ are nonzero numbers. Then
\begin{equation}\label{eq.rttgpmv}
\begin{split}
\sum_{j = 0}^n {\binom{n - s - r}{j - s}\Big\lfloor \frac{j}2 \Big\rfloor } b^{n - j - r} c^{j - s} &= \frac{(n-r)c+bs}{2}(b + c)^{n - r - s - 1}  \\ 
&\quad\, - \frac14\left( {(b + c)^{n - r - s}  - ( - 1)^s (b - c)^{n - r - s} } \right).
\end{split}
\end{equation}
\end{corollary}
\begin{proof}
Differentiate  \eqref{main_bin_id} $r$ times with respect to $b$ and $s$ times with respect to $c$.
\end{proof}

In view of \eqref{eq.rttgpmv}, it is obvious that Theorems \ref{thm_bin_Fib} to \ref{thm5_bin_Fib} allow further generalizations. 
\begin{theorem}
Let $n$, $r$ and $s$ be non-negative integers. Then
\begin{gather*}
\sum_{j = 1}^n {\binom{n - s - r}{j - s}\Big\lfloor \frac{j}{2} \Big\rfloor } = 2^{n - s - r - 2} (n + s - r - 1),\qquad n > r + s,\\
\sum_{j = 1}^n {(-1)^j\binom{n - s - r}{j - s}\Big\lfloor \frac{j}{2} \Big\rfloor } = 2^{n - s - r - 2},\qquad n \ge r + s+2.
\end{gather*}
\end{theorem}
\begin{proof}
Evaluate \eqref{eq.rttgpmv} at $b=c$ and $b=-c$, respectively.
\end{proof} 
\begin{corollary}
Let $n$, $r$ and $s$ be non-negative integers such that $n\ge r + s+2$. Then
\begin{gather*}
\sum_{j = 1}^{\left\lfloor {n/2} \right\rfloor } {\binom{n - s - r}{2j - s}j}  = 2^{n - s - r - 3} (n + s - r),\\
\sum_{j = 1}^{\left\lfloor {n/2} \right\rfloor} \binom{n - s - r}{2j - s + 1} j = 2^{n - s - r - 3} (n + s - r - 2).
\end{gather*}
\end{corollary}
\begin{lemma}\label{lem.ee72u15}
If $p$, $q$, $r$ and $s$ are rational numbers, then
\[p + q\sqrt 5 = r + s\sqrt 5\iff p=r,\quad q=s.
\]
\end{lemma}
\begin{theorem}\label{thm.cfmxf53}
Let $n$, $r$ and $s$ be non-negative integers with $n\geq r\geq s$. If $s$ is odd, then 
\begin{gather}\label{eq.ca9zow5}
\sum_{j = 1}^{\left\lfloor {n/2} \right\rfloor } {\binom{n -  r}{2j - s}}5^jj 
= 2^{n - r - 3} 5^{\frac{s + 1}{2}} \big( {(n - r)L_{n - r - 1}  + 2sF_{n - r} }\big),\\
\label{eq.xux6fze}
\sum_{j = 1}^{\left\lfloor {n/2} \right\rfloor} {\binom{n - r}{2j + 1 - s}}5^j j 
= 2^{n - r  - 3} 5^{\frac{s - 1}{2}} \big( {5(n - r)F_{n - r - 1}  + 2(s-1)L_{n - r} } \big),
\end{gather}
while if $s$ is even, then 
\begin{gather}\label{eq.hl2fr0p}
\sum_{j = 1}^{\left\lfloor {n/2} \right\rfloor } {\binom{n -  r}{2j - s}}5^jj= 2^{n - r - 3} 5^{\frac{s}2} \big( {5(n - r)F_{n - r - 1}  + 2sL_{n - r} } \big),\\ 
\label{eq.i0t3vw8}
\sum_{j = 1}^{\left\lfloor {n/2} \right\rfloor} {\binom{n - r}{2j + 1 - s}}5^{j}j = 2^{n - r  - 3} 5^{\frac{s}{2}} \big( {(n - r)L_{n - r - 1}  + 2(s-1)F_{n - r } } \big).
\end{gather}
\end{theorem}
\begin{proof}
Choosing $b=1/2$, $c=\sqrt 5/2$ in \eqref{eq.rttgpmv} and assuming $s$ is an odd integer gives
\begin{align*}
\frac{\sqrt 5}{2^{n - r - s}}\sum_{j = 0}^{\left\lfloor {n/2}  \right\rfloor } \binom{n - s - r}{2j - s} 5^jj & + \frac1{2^{n - r - s}}\sum_{j = 1}^{\left\lceil {n/2} \right\rceil } {\binom{n - s - r}{2j - 1 - s}}5^j(j - 1)\\
&=\frac{\sqrt5(n - s - r)}{4}\alpha ^{n - r - s - 1}  + \frac{s}{2}\alpha ^{n - r - s} - \frac14L_{n - r - s} .
\end{align*}
Use of $2\alpha^m=L_m + F_m\sqrt 5$ in the above identity and application of Lemma \ref{lem.ee72u15} yield \eqref{eq.ca9zow5} and \eqref{eq.xux6fze}. The proof of \eqref{eq.hl2fr0p} and~\eqref{eq.i0t3vw8} is similar.
\end{proof}

\section{More series and identities}

From the general relations \eqref{gf_floor_id1} and \eqref{bin_gen_fkt} it is obvious that many particular examples 
can be studied. For instance, choosing $a_n=n^2$ we get the next corollary.
\begin{corollary}\label{example2}
For $k\geq 1$ and $z\in\mathbb{C}$ with $|z|<1$ the following expressions are valid:
\begin{gather*}
\sum_{n=0}^\infty \Big\lfloor \frac{n}{k} \Big\rfloor^2 z^n = \frac{z^k(1+z^k)}{(1-z)(1-z^k)^2},\\
\sum_{n=0}^\infty (-1)^n \Big\lfloor \frac{n}{k} \Big\rfloor^2  z^n = (-1)^k \frac{z^k (1+(-1)^k z^k)}{(1+z) (1+(-1)^{k+1}z^k)^2}.
\end{gather*}
\end{corollary}
\begin{proof}
Use Lemma \ref{fund_lem} with $a_n=n^2$ in conjunction with 
$
F(z) = \sum\limits_{n=0}^\infty n^2 z^n = \frac{z(1+z)}{(1-z)^3}$. The second identity follows from the first by replacing $z$ with $-z$.
\end{proof}

Corollary \ref{example2} also leads to new Fibonacci (Lucas) series evaluations, one of which is stated in the next theorem. 
\begin{theorem}
For $k\geq 1$ and $m$ integers, we have
\begin{equation*}
\sum_{n=0}^\infty \Big\lfloor \frac{n}{k} \Big\rfloor^2 \frac{F_{n+m-2}}{2^{n+1}} 
= \frac{4^k F_{m+2k} + 2^{k+1}(2^{2k-1}-(-1)^k)F_{m+k} + 2^kF_{m-k} + (1-2(-4)^k)F_{m}}{\big(4^{k} - 2^{k} L_k + (-1)^k\big)^2},
\end{equation*}
\begin{equation*}
\sum_{n=0}^\infty \Big\lfloor \frac{n}{k} \Big\rfloor^2 \frac{L_{n+m-2}}{2^{n+1}} 
= \frac{4^k L_{m+2k} + 2^{k+1}(2^{2k-1}-(-1)^k)L_{m+k} + 2^kL_{m-k} + (1-2(-4)^k)L_{m}}{\big(4^{k} - 2^{k} L_k + (-1)^k\big)^2}.
\end{equation*}
\end{theorem}

For the binomial transform of $a_{\lfloor n/k \rfloor} = \lfloor {n}/{k} \rfloor^2$, i.e., the sum
\begin{equation*}
u_{n} = \sum_{j=0}^n \binom {n}{j} b^{n-j} c^j \Big\lfloor \frac{n}{k} \Big\rfloor^2
\end{equation*}
with $b$ and $c$ real, the ordinary generating function is given by
\begin{align*}
S_u^* (z,k) =  \frac{ 1-\left (\frac{cz}{1-bz}\right )^k }{1-(b+c)z} 
\frac{\left (\frac{cz}{1-bz}\right )^k \left ( 1 + \left (\frac{cz}{1-bz}\right )^k \right )}{\left ( 1 - \left (\frac{cz}{1-bz}\right )^k \right )^3} =  \frac{\left (\frac{cz}{1-bz}\right )^k \left ( 1 + \left (\frac{cz}{1-bz}\right )^k \right )} 
{(1-(b+c)z) \left ( 1 - \left (\frac{cz}{1-bz}\right )^k \right )^2}.
\end{align*}

To highlight the increasing algebraic complexity, we again consider only the case $k=2$. In this particular case the ordinary  generating function reduces to
\begin{equation*}
S_u^* (z,2) = \frac{(cz)^2 \big((1-bz)^2 + (cz)^2\big)}{(1-(b+c)z)\big((1-bz)^2 - (cz)^2\big )^2}.
\end{equation*}
From this expression we can prove that the following result.
\begin{theorem}\label{thm_bin2}
For nonzero real numbers $b$ and $c$ and $n\geq 0$, we have
\begin{equation}\label{main_bin_id2}
\sum_{j=1}^n \binom {n}{j} \Big\lfloor \frac{j}{2} \Big\rfloor^2 b^{n-j} c^j  
= \frac{c^2n (n - 1)(b + c)^{n - 2}}{4} + \frac{(b + c)^n - (b - c)^n}{8} - \frac{cn (b - c)^{n - 1}}{4}.
\end{equation}
\end{theorem}
\begin{proof} We write
\begin{equation*}
S_u^* (z,2) = \frac{(cz)^2 \big((1-bz)^2 + (cz)^2\big)}{(1-(b+c)z)^3 (1-(b-c)z)^2}
\end{equation*}
and use the power series
$\frac{1}{(1-wz)^2} = \sum\limits_{n=0}^\infty (n+1)w^n z^n$ and $
\frac{2}{(1-wz)^3} = \sum\limits_{n=0}^\infty  (n+1)(n+2) w^n z^n
$
to get
\begin{align*}
\frac{2}{(1-(b+c)z)^3 (1-(b-c)z)^2} &\\=  \sum_{n=0}^\infty \sum_{j=0}^n & (j+1)(b-c)^j (n-j+1)(n-j+2)(b+c)^{n-j} z^n. 
\end{align*}
This shows that
\begin{align*}
S_u^*(z,2) & = \frac{c^2}{2} \sum_{n=0}^\infty \sum_{j=0}^n (j+1)(b-c)^j (n-j+1)(n-j+2)(b+c)^{n-j} z^{n+2} \\
&\quad\, - bc^2 \sum_{n=0}^\infty \sum_{j=0}^n (j+1)(b-c)^j (n-j+1)(n-j+2)(b+c)^{n-j} z^{n+3} \\
&\quad\, + \frac{b^2c^2+c^4}{2} \sum_{n=0}^\infty \sum_{j=0}^n (j+1)(b-c)^j (n-j+1)(n-j+2)(b+c)^{n-j} z^{n+4} \\
& =  c^2 z^2 + (3bc^2+c^3)z^3 + \sum_{n=4}^\infty \big(A(n)+B(n)\big) z^n
\end{align*}
with
\begin{align*}
A(n) &= c^2 (n-1) (b-c)^{n-2} + 3c^2 (n-2) (b-c)^{n-3} (b+c) - 2bc^2 (n-2) (b-c)^{n-3},\\
B(n) & =  \sum_{j=0}^{n-4} (j+1)(b-c)^j (b+c)^{n-4-j} \Big ( \frac{c^2}{2} (b+c)^2 (n-j)(n-1-j) \nonumber \\
&\quad\, - bc^2 (b+c)(n-1-j)(n-2-j) + \frac{1}{2}(b^2c^2+c^4)(n-2-j)(n-3-j) \Big ) \nonumber\\
& =  \frac{c^2}{2} (b+c)^{n-2} \sum_{j=0}^{n-4} (j+1)(n-j)(n-1-j) q^j  \nonumber\\
&\quad\, - bc^2 (b+c)^{n-3} \sum_{j=0}^{n-4} (j+1)(n-1-j)(n-2-j) q^j \nonumber \\
&\quad\, + \frac{1}{2}(b^2c^2+c^4) (b+c)^{n-4} \sum_{j=0}^{n-4} (j+1)(n-2-j)(n-3-j) q^j, 
\end{align*}
where $q=(b-c)/(b+c)$. After simplifying $A(n)$ this shows that
\begin{align*}
\sum_{j=1}^n \binom {n}{j} \Big\lfloor \frac{j}{2} \Big\rfloor^2 b^{n-j} c^j  = \begin{cases}
c^2, & n=2; \\ 
3bc^2 + c^3, & n=3; \\
c^2 (b-c)^{n-3} (b+c)(2n-3) - 2c^3(b-c)^{n-3} + B(n), & n\geq 4. 
\end{cases} 
\end{align*}

From here we use the convolution identities 
\begin{align*}
& \sum_{j=0}^{n-4} (j+1)(n-j)(n-1-j) z^j = \frac{1}{(z-1)^4 z^3} \Big ( (n^2+3n+2) z^5 - (2n^2+2n-4)z^4 \\
& \qquad+ (n^2-n)z^3 - (6n-12)z^n + (22n-46)z^{n+1} - (28n-64)z^{n+2} + (12n-36)z^{n+3} \Big ),\\
& \sum_{j=0}^{n-4} (j+1)(n-1-j)(n-2-j) z^j = \frac{1}{(z-1)^4 z^3} \Big ( (n^2+n) z^5 - (2n^2-2n-4)z^4 \\
& \qquad+ (n^2-3n+2)z^3 - (2n-4)z^n + (8n-16)z^{n+1} - (12n-24)z^{n+2} + (6n-18)z^{n+3} \Big )
\end{align*}
and
\begin{align*}
&\sum_{j=0}^{n-4} (j+1)(n-2-j)(n-3-j) z^j \\ 
&\qquad=  \frac{1}{(z-1)^4 z^3} \Big ( (n^2-5n+6) z - (2n^2-6n)z^2 + (n^2-n)z^3 - 2n z^n + (2n-6)z^{n+1} \Big ),
\end{align*}
insert into $B(n)$ and simplify to get the intimidating expression
\begin{equation*}
\sum_{j=1}^n\binom {n}{j}\Big\lfloor \frac{j}{2} \Big\rfloor^2  b^{n-j} c^j  = \begin{cases}
c^2, & n=2; \\ 
3bc^2 + c^3, & n=3;\\
S(n,b,c,q), & n\geq 4,
\end{cases} 
\end{equation*}
where
\begin{align*}
& S(n,b,c,q) = c^2 (b-c)^{n-3} (b+c)(2n-3) - 2c^3(b-c)^{n-3} \\
&\qquad+ 8 c^6 (b+c)^{n-9} (b-c)^3 \big ( (n^2+3n+2)q^5 - (2n^2+2n-4)q^4 \\
&\qquad+ (n^2-n) q^3 - (6n-12) q^n + (22n-46)q^{n+1} - (28n-64)q^{n+2} + (12n-36)q^{n+3} \big ) \\
&\qquad- 16 b c^6 (b+c)^{n-10} (b-c)^3 \big ( (n^2+n)q^5 - (2n^2-2n-4)q^4 + (n^2-3n+2) q^3\\
&\qquad - (2n-4) q^n + (8n-16)q^{n+1} - (12n-24)q^{n+2} + (6n-18)q^{n+3} \big ) \\
&\qquad+ 8c^4 (b^2c^2+c^4)(b+c)^{n-11} (b-c)^3 \big ( (n^2-5n+6)q - (2n^2-6n)q^2 \\
&\qquad+ (n^2-n) q^3 - 2n q^n + (2n-6)q^{n+1} \big ) 
\end{align*}
with $q=(b-c)/(b+c)$. 

Additional simplifications reduce the above expression to the stated form.
\end{proof}
\begin{remark}
A much shorter and more elegant proof of Theorem \ref{thm_bin2} not relying on generating functions goes as follows. 
Let
\[
h_1 (x;b,c,n) = \sum_{j = 0}^{\left\lfloor {n/2} \right\rfloor } \binom {n}{2j} b^{n - 2j} c^{2j} x^{2j} 
= \frac{(b + cx)^n + (b - cx)^n }{2},
\]
\[
h_2 (x;b,c,n) = \sum_{j = 1}^{\left\lceil {n/2} \right\rceil } \binom {n}{2j-1} b^{n - 2j + 1} c^{2j - 1} x^{2j - 2} 
= \frac{(b + cx)^n - (b - cx)^n}{2x}.
\]
Then
\[
4\sum_{j = 1}^n \binom {n}{j}\Bigl\lfloor {\frac{j}{2}} \Bigr\rfloor ^2 b^{n - j} c^j  = 
\left. {\frac{d}{{dx}}\left( {x\frac{{dh_1 }}{dx}} \right)} \right|_{x = 1}  
+ \left. {\frac{d}{{dx}}\left( {x\frac{{dh_2 }}{dx}} \right)} \right|_{x = 1}
\]
from which the result follows. This procedure also gives Theorem \ref{thm_bin} as
\[
2\sum_{j = 1}^n \binom {n}{j} \Bigl\lfloor {\frac{j}{2}} \Bigr\rfloor b^{n - j} c^j  
= \left. {\frac{{dh_1 }}{dx}} \right|_{x = 1} + \left. {\frac{{dh_2 }}{dx}} \right|_{x = 1}.
\]
Of course, sums involving higher powers of the floor function can be computed in this manner.
\end{remark}

As an example of a binomial sum with Fibonacci (Lucas) numbers and $\big\lfloor {j}/{2} \big\rfloor^2$ we can state the next theorem.
We invite the interested readers to derive many more identities of this kind.
\begin{theorem}
For $m$ an integer and $n$ a non-negative integer, we have
\begin{equation*}
\sum_{j=1}^n \binom {n}{j}  \Big\lfloor \frac{j}{2} \Big\rfloor^2 F_{j+m}
= \frac{n(n-1)}{4} F_{2n+m-2} + \frac{1}{8} \big (F_{2n+m} + (-1)^m F_{n-m} \big) - \frac{(-1)^{m}n}{4} F_{n-2-m},
\end{equation*}
\begin{equation*}
\sum_{j=1}^n \binom {n}{j}  \Big\lfloor \frac{j}{2} \Big\rfloor^2 L_{j+m}
= \frac{n(n-1)}{4} L_{2n+m-2} + \frac{1}{8} \big (L_{2n+m} - (-1)^{m} L_{n-m} \big) + \frac{(-1)^{m}n}{4} L_{n-2-m}.
\end{equation*}
\end{theorem}
\begin{corollary} We have
\begin{align*}
\sum_{j=1}^n &\binom {n}{j} \Big\lfloor \frac{j}{2} \Big\rfloor^2 = 
\begin{cases}
0, & n=1; \\ 
2^{n-4}(n^2-n+2), & n\geq 2, 
\end{cases} \\
\sum_{j=1}^n &(-1)^j  \binom {n}{j} \Big\lfloor \frac{j}{2} \Big\rfloor^2 = 
\begin{cases}
0, & n=1; \\ 
1, & n=2; \\
2^{n-3}(n-1), & n\geq 3. 
\end{cases} 
\end{align*}
\end{corollary}
\begin{proof}
Set $b=c$ in \eqref{main_bin_id2} in conjunction with the convention $0^0=1$. 
For the second identity work with $b=1$ and $c=-1$ in \eqref{main_bin_id2}. 
\end{proof}
\begin{corollary}
Let $n$, $r$ and $s$ be non-negative integers such that $n\geq r\geq s$.  
Let $b$ and $c$ be nonzero real numbers. Then
\begin{equation}\label{eq.f7u15am}
\begin{split}
\sum_{j = 1}^n \binom{n - r }{j - s}&b^{n - j - r+s} c^{j - s} \Big\lfloor \frac{j}{2} \Big\rfloor ^2  = \frac{{(n - r - 1)(n - r)}}{4} c^2 (b + c)^{n - r - 2}\\ 
&\qquad\qquad\qquad + \frac{{(n - r)c}}{4}\left( {2s(b + c)^{n - r - 1}  - ( - 1)^s (b - c)^{n - r - 1} } \right)\\
&\qquad\qquad\qquad + \frac{{2s(s - 1) + 1} }{8}(b + c)^{n - r }  + \frac{{( - 1)^s (2s - 1)}}{8}(b - c)^{n - r}. 
\end{split}
\end{equation}
\end{corollary}
\begin{proof}
Differentiate \eqref{main_bin_id2} $r$ times with respect to $b$ and $s$ times with respect to $c$.
\end{proof}

In particular, we have
\begin{align*}
\sum_{j = 1}^n & \binom{n - r }{j - s}\Big\lfloor \frac{j}{2} \Big\rfloor ^2 =
2^{n - r - 4} \left(\Big(n + 2s -r-\frac12\Big)^2 +\frac74 - 2s\right), \quad n-2 \ge r \ge s,\\
&\sum_{j = 1}^n  {(-1)^j\binom{n - r}{j - s}\Big\lfloor \frac{j}{2} \Big\rfloor ^2 }=2^{n  - r - 3}\left(n + 2s - r - 1\right),\quad n-2 > r\geq  s,
\end{align*}
with the special values, for $n\geq2s+2$,
\begin{gather*}
\sum_{j = 1}^{2n - s}\binom{n - 2s}{j - s}\Big\lfloor \frac{j}{2} \Big\rfloor ^2 = 2^{n - 2s - 4}\left(n^2 - n - 2s + 2\right),
\\
\sum_{j = 1}^n (-1)^j\binom{n - 2s + 1}{j - s}\Big\lfloor \frac{j}{2} \Big\rfloor ^2 = 2^{n - 2s - 2}n.
\end{gather*}
\begin{corollary}
Let $n$, $r$ and $s$ be non-negative integers such that $n-2>r\geq s$.
 Then
\begin{gather*}
\sum_{j = 1}^{\left\lfloor {n/2} \right\rfloor } \binom{n - r}{2j - s} j^2  = 2^{n - r - 5} \left((n -r+ 2s)^2  + n-r \right),\\
\sum_{j = 1}^{\left\lfloor {n/2} \right\rfloor } \binom{n - r}{2j - s + 1} j^2  = 2^{n - r - 5} \left( {(n -r + 2s)^2  
	-3(n-r)-8s+4} \right).
\end{gather*}
\end{corollary}
\begin{theorem}
Let $n$, $r$ and $s$ be non-negative integers such that $n\ge r\ge s$. If $s$ is odd, then 
\begin{align*}
\sum_{j = 1}^{\left\lfloor {n/2} \right\rfloor} \binom{n - r}{2j - s}5^j j^2   &= 2^{n - r - 5}\, 5^{(s + 1)/2} \big( {5(n - r - 1)(n - r)F_{n - r  - 2} } \\
&\quad  + 2(2s + 1)(n -  r)L_{n - r - 1}  + 4s^2 F_{n - r } \big),\\
\sum_{j = 1}^{\left\lfloor {n/2} \right\rfloor} \binom{n - r }{2j - s + 1} 5^j j^2  &= 2^{n - r - 5}\, 5^{(s - 1)/2} \big( {5(n -  r - 1)(n - r)L_{n - r - 2} } \\
&\quad { + 10\,(2s - 1)(n - r)F_{n - r - 1}  + 4(s - 1)^2 L_{n - r } } \bigr);
\end{align*}
while, if $s$ is even, then
\begin{align*}
\sum_{j = 1}^{\left\lfloor {n/2} \right\rfloor } \binom{n - r }{2j - s}5^j j^2   &= 2^{n - r - 5}\, 5^{s/2} \big( {5(n -  r - 1)(n - r)L_{n - r  - 2} } \\
&\quad  + 10\,(2s + 1)(n - r)F_{n - r - 1}  + 4s^2 L_{n - r }  \big),\\
\sum_{j = 1}^{\left\lfloor {n/2} \right\rfloor } \binom{n - r }{2j - s + 1}5^{j} j^2  &= 2^{n - r - 5}\, 5^{s/2} \big( {5(n - r - 1)(n - r)F_{n - r - 2} }\\
&\quad\,  + 2\,(2s - 1)(n - r)L_{n - r - 1}  + 4(s-1)^2 F_{n - r}  \big).
\end{align*}
\end{theorem} 
\begin{proof}
Set $b=1/2$ and $c=\sqrt5/2$ in \eqref{eq.f7u15am} and proceed as in the proof of Theorem \ref{thm.cfmxf53}.
\end{proof}

\section{A final example}

Another interesting family of infinite series  comes from choosing $a_n=(-1)^n$ 
or $a_{\lfloor n/k \rfloor} = (-1)^{\lfloor n/k \rfloor}$. As in this case 
$$
F(z)=\frac{1}{1+z},\quad |z|<1,
$$
we get immediately from Lemma \ref{fund_lem}
\begin{equation}\label{F+}
F^+(z,k) = \sum_{n=0}^\infty (-1)^{\lfloor n/k \rfloor} z^n = \frac{1-z^k}{(1-z)(1+z^k)}, \quad k\ge1.
\end{equation}

This leads to the next theorem.
\begin{theorem}
For integers $k\ge1$, $p\ge2$ and any integer $m$, we have
\begin{equation}\label{Final1}
\begin{split}
&\sum_{n=0}^{\infty} (-1)^{\lfloor n/k \rfloor} \frac{F_{n+m}}{p^{n+1}} \\
&= \frac{(p^{2k}-(-1)^k)(pF_m+F_{m-1})-p^k(pF_{k+m}+F_{k+m-1})+(-1)^mp^k(F_{k-m+1}-pF_{k-m})}{(p^2-p-1)\big(p^{2k}+p^kL_k+(-1)^k\big)},
\end{split}
\end{equation}
\begin{equation}\label{Final2}
\begin{split}
&\sum_{n=0}^{\infty} (-1)^{\lfloor n/k \rfloor} \frac{L_{n+m}}{p^{n+1}} \\
&=\frac{(p^{2k}-(-1)^k)(pL_m+L_{m-1})-p^k(pL_{k+m}+L_{k+m-1})-(-1)^m p^k(L_{k-m+1}- pL_{k-m})}{(p^2-p-1)\big(p^{2k}+p^kL_k+(-1)^k\big)}.
\end{split}
\end{equation}
\end{theorem}
\begin{proof}
Use of $z=\alpha/p$ and  $z=\beta/p$ with $p\geq2$, in turn, in \eqref{F+}. Then difference and additions the resulting identities gives  of 
\eqref{Final1} and \eqref{Final2}, respectively.
\end{proof}
\begin{example} 
For integers $k\ge1$ and any integer $m$,
\begin{align*}
\sum_{n=0}^{\infty} (-1)^{\lfloor n/k \rfloor} \frac{F_{n+m}}{2^{n+1}} &= \frac{(4^{k}-(-1)^k)F_{m+2}-2^k F_{k+m+2} -(-1)^m 2^kF_{k-m-2}}{4^{k}+2^k L_k+(-1)^k},\\
\sum_{n=0}^{\infty} (-1)^{\lfloor n/k \rfloor} \frac{L_{n+m}}{2^{n+1}} &= \frac{(4^{k}-(-1)^k)L_{m+2}-2^k L_{k+m+2} + (-1)^m 2^kL_{k-m-2}}{4^{k}+2^k L_k+(-1)^k},\\
\sum_{n=0}^{\infty} (-1)^{\lfloor n/k \rfloor} \frac{F_{n+m}}{3^{n+1}} &= \frac{(9^{k}-(-1)^k)L_{m+1}-3^k L_{k+m+1} -(-1)^m 3^k L_{k-m-1}}{5(9^{k}+3^k L_k+(-1)^k)},\\
\sum_{n=0}^{\infty} (-1)^{\lfloor n/k \rfloor} \frac{L_{n+m}}{3^{n+1}} &= \frac{(9^{k}-(-1)^k)F_{m+1}-3^k F_{k+m+1} {+} (-1)^m 3^k F_{k-m-1}}{9^{k}+3^k L_k+(-1)^k}.
\end{align*}
\end{example}

For the finite binomial sum
\begin{equation*}
v_n = \sum_{j=0}^n(-1)^{\lfloor j/k \rfloor}  \binom {n}{j} b^{n-j} c^j 
\end{equation*}
with $b$ and $c$ real, the ordinary generating function is given by
\begin{align*}
S_v^* (z,k)  = \frac{ 1-\left (\frac{cz}{1-bz}\right )^k }{1-(b+c)z} \frac{1}{1 + \left (\frac{cz}{1-bz}\right )^k} = \frac{1}{1-(b+c)z}\left ( 1 -  \frac{2(cz)^k}{(1-bz)^k + (cz)^k} \right).
\end{align*}

Now the same analysis as above applies.  
\begin{theorem}\label{thm_bin3}
For nonzero real numbers $b$ and $c$ and $n\geq 2$ we have
\begin{equation}\label{main_bin_id3}
\sum_{j=0}^n (-1)^{\lfloor j/2 \rfloor} \binom {n}{j} b^{n-j} c^j = \frac{1}{2} (1-i)(b+ic)^{n} + \frac{1}{2} (1+i) (b-ic)^{n},
\end{equation}
where $i$ is the imaginary unit. Equivalently, using the polar form, we get the expression
\begin{equation*} 
\sum_{j=0}^n (-1)^{\lfloor {j/2} \rfloor}\binom {n}{j} b^{n-j} c^j  = \sgn b\,
\sqrt{2(b^2 + c^2 )^n} \cos \left( {n\arctan  {\frac{c}{b}} - \frac{\pi }{4}} \right).
\end{equation*}
\end{theorem}
\begin{proof}
As
\begin{equation*}
\frac{1}{(1-bz)^2 + (cz)^2} = \frac{1}{2c} \sum_{n=0}^\infty \big((b+ic)^n (c-ib) + (b-ic)^n (c+ib)\big) z^n,
\end{equation*}
we have
\begin{align*}
S_v^* (z,2) & =  \sum_{n=0}^\infty (b+c)^n z^n - c z^2 \sum_{n=0}^\infty (b+c)^n z^n  \sum_{n=0}^\infty \left ((b+ic)^n (c-ib) + (b-ic)^n (c+ib)\right ) z^n \\
& =  (b+c)^0 + (b+c)^1 z \\
& \quad + \sum_{n=2}^\infty \left ( (b+c)^n - c \sum_{j=0}^{n-2} (b+c)^j \left ( (b+ic)^{n-2-j} (c-ib) + (b-ic)^{n-2-j} (c+ib)\right )\right )\! z^n.
\end{align*}
This shows that
\begin{align*}
\sum_{j=0}^n (-1)^{\lfloor j/2 \rfloor} &\binom {n}{j} b^{n-j} c^j\\  
&= (b+c)^n - c \sum_{j=0}^{n-2} (b+c)^j \left ( (b+ic)^{n-2-j} (c-ib) + (b-ic)^{n-2-j} (c+ib)\right ).
\end{align*}
The last sum can be rewritten as
\begin{align*}
& \sum_{j=0}^{n-2} (b+c)^j \left ( (b+ic)^{n-2-j} (c-ib) + (b-ic)^{n-2-j} (c+ib)\right ) \\
& \quad = (b+ic)^{n-2} (c-ib) \sum_{j=0}^{n-2} \left ( \frac{b+c}{b+ic}\right )^j 
+ (b-ic)^{n-2} (c+ib) \sum_{j=0}^{n-2} \left ( \frac{b+c}{b-ic}\right )^j \\
& \quad = \frac{c-ib}{c(1-i)}\left ( (b+c)^{n-1} - (b+ic)^{n-1}\right ) + \frac{c+ib}{c(1+i)}\left ( (b+c)^{n-1} - (b-ic)^{n-1}\right ) \\
& \quad = \frac{(b+c)^n}{c} - \frac{1}{2} (b+ic)^{n-1} \left ( 1 + \frac{b}{c} + i \left ( 1-\frac{b}{c}\right ) \right ) 
- \frac{1}{2} (b-ic)^{n-1} \left ( 1+\frac{b}{c} - i \left ( 1-\frac{b}{c}\right ) \right ) \\
& \quad = \frac{(b+c)^n}{c} - \frac{b+c}{2c} \left ((b+ic)^{n-1} + (b-ic)^{n-1}\right ) + i \frac{b-c}{2c} \left ((b+ic)^{n-1} - (b-ic)^{n-1}\right ).
\end{align*}
This gives
\begin{align*}
\sum_{j=0}^n (-1)^{\lfloor\frac{j}{2} \rfloor}  \binom {n}{j} b^{n-j} c^j =
\frac{b+c}{2}\! \left ((b+ic)^{n-1} + (b-ic)^{n-1}\right )\! - i \frac{b-c}{2} \left ( (b+ic)^{n-1}\! - (b-ic)^{n-1} \right )\!,
\end{align*}
and from here
\begin{align*}
\sum_{j=0}^n(-1)^{\lfloor j/2 \rfloor}  \binom {n}{j} b^{n-j} c^j & =  \frac{b}{2}(1-i) (b+ic)^{n-1} + \frac{c}{2} (1+i) (b+ic)^{n-1} \\
& \quad + \frac{b}{2}(1+i) (b-ic)^{n-1} + \frac{c}{2} (1-i) (b-ic)^{n-1} \\
& = \frac{(b+ic)^n}{b^2+c^2}\left ( \frac{b}{2}(1-i) (b-ic) + \frac{c}{2} (1+i) (b-ic) \right ) \\
& \quad + \frac{(b-ic)^n}{b^2+c^2}\left ( \frac{b}{2}(1+i) (b+ic) + \frac{c}{2} (1-i) (b+ic) \right )
\end{align*}
which is the stated formula.
\end{proof}
\begin{remark}
It is worth noting that a very short proof of Theorem \ref{thm_bin3}  goes as follows. Recall
\begin{equation*}
h_1 (x;b,c,n) = \sum_{j = 0}^{\left\lfloor {n/2} \right\rfloor } \binom {n} {2j} b^{n - 2j} c^{2j} x^{2j} 
= \frac{(b + cx)^n + (b - cx)^n}{2},
\end{equation*}
\begin{equation*}
h_2 (x;b,c,n) = \sum_{j = 1}^{\left\lceil {n/2} \right\rceil } \binom {n} {2j-1} b^{n - 2j + 1} c^{2j - 1} x^{2j - 2} 
= \frac{(b + cx)^n - (b - cx)^n}{2x}.
\end{equation*}
Therefore,
\begin{align*}
\sum_{j = 0}^n (-1)^{\left\lfloor {j/2} \right\rfloor } \binom {n} {j} b^{n - j} c^j  & =  h_1 (i;b,c,n) + h_2 (i;b,c,n) \\
& =  \frac{1}{2}(1 - i)(b + ic)^n  + \frac{1}{2}(1 + i)(b - ic)^n.
\end{align*} 
\end{remark}
\begin{corollary}
For $n\geq 2$ we have
\begin{gather*}
\sum_{j=0}
^n (-1)^{\lfloor j/2 \rfloor} \binom {n}{j} = 2^{(n+1)/2} \cos \frac{(n-1)\pi}{4},\\
\sum_{j=0}^n (-1)^{\lfloor 3j/2 \rfloor + 1} \binom {n}{j}  = 2^{(n+1)/2} \sin \frac{(n-1)\pi}{4}.
\end{gather*}
\end{corollary}
\begin{proof}
Set $b=c$ in Theorem \ref{thm_bin3} to get
\begin{align*}
\sum_{j=0}^n (-1)^{\lfloor j/2 \rfloor}\binom {n}{j}  & =  (1+i)^{n-1} + (1-i)^{n-1} \\
& =  2^{(n-1)/2} \left ( e^{\pi i (n-1)/4} + e^{-\pi i (n-1)/4} \right ) =  2^{(n-1)/2}  2 \cos \frac{(n-1)\pi}{4},
\end{align*}
as desired. Equivalently, one can use $\arctan 1=\pi/4$. The second identity is proved in the same manner working with $(b,c)=(1,-1)$ 
and using the definition of the complex sine function.
\end{proof}
\begin{corollary}
Let $n$, $r$ and $s$ be non-negative integers such that $n\ge r\ge s$.  
Let $b$ and $c$ be nonzero real or complex numbers. Then
\begin{equation}\label{eq.rf08l60}
\sum_{j = 0}^n {( - 1)^{\left\lfloor {j/2} \right\rfloor } \binom{n - r}{j - s} b^{n - j - r+s} c^{j - s} }  = \frac{{1 - i}}{2}i^s (b + ic)^{n - r }  + \frac{{1 + i}}{2}( - i)^s (b - ic)^{n - r}.
\end{equation}
\end{corollary}
\begin{proof}
Differentiate \eqref{main_bin_id3} $r$ times with respect to $b$ and $s$ times with respect to $c$.
\end{proof}
The polar form of \eqref{eq.rf08l60} is
\begin{align*}
\sum_{j = 0}^n ( - 1)^{\left\lfloor {j/2} \right\rfloor } &\binom{n - r}{j - s} b^{n - j - r+s} c^{j - s} \\
&\quad  = \sgn{b}\,\sqrt{2(b^2  + c^2 )^{n - r } } \cos \left( {(n - r )\arctan \frac{c}{b} + \frac{(2s-1)\pi }{4}} \right).
\end{align*}
\begin{corollary}
Let $n$, $r$ and $s$ be non-negative integers such that $n\ge r \ge s$. 
Then
\begin{gather*}
\sum_{j = 0}^n {( - 1)^{\left\lfloor {j/2} \right\rfloor }\binom{n - r}{j - s} } = 2^{(n  - r + 1)/2} \cos \left( {\frac{\pi }{4}(n + 2s - r - 1)} \right),\\
\sum_{j = 0}^n {( - 1)^{\left\lfloor {3j/2} \right\rfloor}} \binom{n - r}{j - s} = (-1)^s\,2^{(n - r + 1)/2} \cos \left( {\frac{\pi }{4}(n - 2s - r + 1)} \right).
\end{gather*}
\end{corollary}
\begin{corollary}
Let $n$, $r$ and $s$ be non-negative integers such that $n\ge r \ge s$. 
Let $b$ and $c$ be non-zero real numbers. Then
\begin{gather}\label{eq.rg96dg0}
\sum_{j = 0}^{\left\lfloor {n/2} \right\rfloor } {\binom{n - r}{2j - s}b^{n - 2j - r+s} c^{2j - s} } =\frac12\big( (b + c)^{n - r }  + ( - 1)^s (b - c)^{n - r}\big),\\
\label{eq.qmk188v}
\sum_{j = 1}^{\left\lfloor {n/2} \right\rfloor } {\binom{n - r}{2j - s + 1}b^{n - 2j - 1 - r+s} c^{2j + 1 - s} } = \frac{1}{2}\big((b + c)^{n - r }  - ( - 1)^s (b - c)^{n - r}\big). 
\end{gather}
\end{corollary}
\begin{proof}
Assume that $b$ and $c$ are real numbers. Write $ic$ for $c$ in~\eqref{eq.rf08l60} and compare real and imaginary parts of both sides.
\end{proof}
\begin{example}
Choosing $b=1/2$, $c=\sqrt 5/2$ in \eqref{eq.rg96dg0} and \eqref{eq.qmk188v} 
 gives for $n\ge r \ge s$: 
\begin{gather*}
\sum_{j = 1}^{\left\lfloor {n/2} \right\rfloor } \binom{n - r}{2j - s}\,5^j  = 2^{n - r - 1}\, 5^{(s + 1)/2} F_{n - r } ,\\
\sum_{j = 1}^{\left\lfloor {n/2} \right\rfloor } {\binom{n  - r}{2j + 1 - s}\,5^j }  = 2^{n - r  -1}\, 5^{(s - 1)/2} L_{n - r}.
\end{gather*}
\end{example}

\section{Concluding comments}

This paper was about studying infinite series and binomial sums involving a general floor sequence $(a_{\lfloor n/k \rfloor})_{n\geq 0}$ 
which we termed the dual sequence corresponding to $(a_n)_{n\geq 0}$. In the first part we derived some properties of 
infinite series involving $(a_{\lfloor n/k \rfloor})_{n\geq 0}$ and evaluated these series in closed form at certain Fibonacci (Lucas) arguments. As a byproduct we have also established a apparently new expression for the polylogarithm. 

In the second part of the study we have studied finite binomial sums, hereby focusing on three particular classes of sequences, namely, $$a_{\lfloor n/k \rfloor} = \lfloor n/k \rfloor,\quad
a_{\lfloor n/k \rfloor} = \lfloor n/k \rfloor^2 \quad \text{ and } \quad a_{\lfloor n/k \rfloor} = (-1)^{\lfloor n/k \rfloor}.$$
Based on three general theorems (Theorem \ref{thm_bin}, Theorem \ref{thm_bin2} and Theorem \ref{thm_bin3}) a range of new Fibonacci (Lucas) binomial sums were evaluated exactly. As was pointed out in the main text there are still gaps to fill which we leave for the readers' personal exploration. 

We cannot resist to close the article with another obvious class of sequences. 
This class comes from choosing $a_n=F_n$ and $a_n=L_n$, respectively. Using well-known generating functions  
$$\sum_{n=0}^{\infty}F_n z^n=\frac{z}{1-z-z^2}\quad \text{and} \quad \sum_{n=0}^{\infty}L_n z^n=\frac{2-z}{1-z-z^2}, $$
we get from Lemma \ref{fund_lem} for integers $k\ge1$ and $|z|<|\beta|$
\begin{equation*}
\sum_{n=0}^\infty F_{\lfloor n/k \rfloor} z^n = \frac{(1-z^k)z^k}{(1-z)(1-z^k-z^{2k})},\qquad
\sum_{n=0}^\infty L_{\lfloor n/k \rfloor} z^n = \frac{(1-z^k)(2-z^k)}{(1-z)(1-z^k-z^{2k})}.
\end{equation*}

This produces the next results.
\begin{theorem}
For integers $k\ge1$, $p\ge2$ and integer $m$, we have
\begin{align*}
D_k(p)\sum_{n=0}^{\infty} & \frac{F_{\lfloor n/k \rfloor}F_{n+m}}{p^{n+1}} =p^{3k+1}F_{k+m}+p^{3k}F_{k+m-1}\\
	&\quad -p^{2k+1}\big((-1)^kF_{m}+F_{2k+m}\big) - p^{2k}\big((-1)^kF_{m-1}+F_{2k+m-1}\big)\\
	&\quad +(-1)^kp^{k+1}\big((-1)^mF_{k-m}+F_{k+m}\big)-(-1)^k p^k \big((-1)^m F_{k-m+1}-F_{k+m-1}\big)\\
	&\quad + pF_{m} + F_{m-1},\\
D_k(p)\sum_{n=0}^{\infty}&  \frac{F_{\lfloor n/k \rfloor}L_{n+m}}{p^{n+1}} =p^{3k+1}L_{k+m}+p^{3k}L_{k+m-1}\\
&\quad -p^{2k+1}((-1)^kL_{m}+L_{2k+m}) - p^{2k}((-1)^kL_{m-1}+L_{2k+m-1})\\
&\quad- (-1)^{k}p^{k+1}((-1)^mL_{k-m} - L_{k+m})+(-1)^k p^k ((-1)^m L_{k-m+1}+L_{k+m-1})\\
&\quad + pL_{m} + L_{m-1},\\
D_k(p)\sum_{n=0}^{\infty} & \frac{L_{\lfloor n/k \rfloor}F_{n+m}}{p^{n+1}} =2p^{4k+1} F_{m} + 2p^{4k} F_{m-1} + p^{3k+1} \big(2(-1)^mF_{k-m}-3F_{k+m}\big)\\
&\quad - p^{3k} \big(3F_{k+m-1}+2(-1)^m F_{k-m+1}\big)
+ p^{2k+1}\big(F_{2k+m} + 2(-1)^m F_{2k-m}+3(-1)^kL_{m}\big)\\ 
&\quad + p^{2k} \big(F_{2k+m-1}+3(-1)^kF_{m-1}-2(-1)^m F_{2k-m+1}\big)\\
&\quad + (-1)^k p^{k+1} \big(3(-1)^{m}F_{k-m} + F_{k+m}\big)- (-1)^k p^k \big(F_{k+m-1}-3(-1)^{m} F_{k-m+1}\big)\\
&\quad -  pF_{m} - F_{m-1},\\
D_k(p)\sum_{n=0}^{\infty} & \frac{L_{\lfloor n/k \rfloor}L_{n+m}}{p^{n+1}} =2p^{4k+1} L_{m} + 2p^{4k} L_{m-1} - p^{3k+1} (3L_{k+m}+2(-1)^mL_{k-m})\\
&\quad - p^{3k} (3L_{k+m-1}-2(-1)^m L_{k-m+1}) + p^{2k+1}\big(L_{2k+m} - 2(-1)^m L_{2k-m}+3(-1)^kL_{m}\big)\\
&\quad + p^{2k} \big(L_{2k+m-1}+3(-1)^kL_{m-1}+2(-1)^m L_{2k-m+1}\big)\\
&\quad + (-1)^k p^{k+1} \big(3(-1)^{m}L_{k-m} - L_{k+m}\big)- (-1)^k p^k \big(3(-1)^{m} L_{k-m+1}+L_{k+m-1}\big)\\
&\quad-  pL_{m} - L_{m-1},
\end{align*}
where $D_k(p)=(p^2-p-1)\big(p^{4k}-p^{3k}L_k-p^{2k}(L_{2k}-(-1)^k)+(-1)^k p^k L_k+1\big)$.
\end{theorem}


\begin{thebibliography}{99}
{
\bibitem{Boyadzhiev}
K.~N. Boyadzhiev, \emph{Notes on the Binomial Transform: Theory and Table with Appendix on Stirling Transform}, World Scientific, 2018.
\vspace{-0.1cm}
\bibitem{CarFer}
L.~Carlitz and H.~H. Ferns, Some Fibonacci and Lucas identities, \textit{Fibonacci Quart.} {\bf8}(1) (1970), 61--73.
\vspace{-0.1cm}
\bibitem{Frontczak}
R.~Frontczak, Problem B-1307, \textit{Fibonacci Quart.} {\bf60}(2) (2022), 177.
\vspace{-0.1cm}
\bibitem{Furdui}
O.~Furdui and A.~S\^int\u{a}m\u{a}rian, Exotic series with fractional part function, \textit{Gaz. Mat. Ser. A} {\bf35}(3--4) (2017), 1--12. 
\vspace{-0.1cm}
\bibitem{Graham}
R.~L. Graham, D.~E. Knuth and O.~Patashnik, \emph{Concrete Mathematics: A Foundation for Computer Science}, Addison-Wesley, 1994.
\vspace{-0.1cm}
\bibitem{Koshy} 
T.~Koshy, \emph{Fibonacci and Lucas Numbers with Applications}, Wiley-Interscience, 2001.
\vspace{-0.1cm}
\bibitem{lewin81}
L.~Lewin, \emph{Polylogarithms and Associated Functions}, Elsevier/North-Holland, 1981.
\vspace{-0.1cm}
\bibitem{Nyblom} 
M.~A. Nyblom, Some curious sequences involving floor and ceiling functions, 
\textit{Amer. Math. Monthly} {\bf109}(6) (2002), 559--564. 
\vspace{-0.1cm}
\bibitem{Palka}
R.~Palka and Z.~Trukszyn, Finite sums of the floor functions series, \textit{Integers} {\bf20} (2020), \#41.
\vspace{-0.1cm}
\bibitem{Podrug}
L.~Podrug and D.~Svrtan, {Some refinements of formulae involving floor and ceiling functions}, 
Preprint, arXiv:~2302.01900v1, 2023, 13 p.
\vspace{-0.1cm}
\bibitem{Prod}
H. Prodinger, {Some information about the binomial transform}, \textit{Fibonacci Quart.} {\bf32}(5) (1994), 412--415.
\vspace{-0.1cm}
\bibitem{Prud}
A.~P. Prudnikov, Y.~A. Brychkov and O.~I. Marichev, \emph{Integrals and Series, Vol.~1}, CRC Press, 1998.
\vspace{-0.1cm}
\bibitem{Shah-1}
D.~Shah, M.~Sahni, R.~Sahni, E.~Le\'on-Castro and M.~Olazabal-Lugo, {Series of floor and ceiling function--Part I: Partial summations}, \textit{Mathematics} {\bf10}(7) (2022), 1178. 
\vspace{-0.1cm}
\bibitem{Shah-2}
D.~Shah, M.~Sahni, R.~Sahni, E.~Le\'on-Castro and M.~Olazabal-Lugo, {Series of floor and ceiling function--Part II: Infinite series}, \textit{Mathematics} {\bf10}(9) (2022), 1566.
\vspace{-0.1cm}
\bibitem{Somu}
S.~T. Somu and A.~Kukla, {On some generalizations to floor function identities of Ramanujan}, \textit{Integers} {\bf22} (2022), \#A33. 
\vspace{-0.1cm}
\bibitem{Vajda} 
S.~Vajda, \emph{Fibonacci and Lucas Numbers, and the Golden Section: Theory and Applications}, Dover Press, 2008.
\vspace{-0.1cm}
\bibitem{Vinson}
J.~Vinson, {The relation of the period modulo $m$ to the rank of apparition of $m$ in the Fibonacci sequence}, \textit{Fibonacci Quart.} {\bf1}(2) (1963), 37--45.
}

\end{thebibliography}
\end{document}